\documentclass[12pt,reqno]{article}
\usepackage{colordvi}
\usepackage{amsmath,amsfonts,amsthm,amssymb}
\usepackage{graphicx}
\usepackage{verbatim}
\usepackage{color}
\usepackage{amscd}
\usepackage{verbatim}
\usepackage{graphicx}
\usepackage{url}

\usepackage{setspace}

\usepackage[left=1.25in, right=1.25in, top=1in, bottom=1in]{geometry}
\usepackage{longtable}

\usepackage{CJKutf8}

\usepackage{amsmath,amssymb,amsthm,color,mathrsfs}
\usepackage{enumitem,anysize}%
\usepackage{verbatim}
\usepackage{amssymb}
\usepackage{tikz}
\usepackage{pgfplots}
\marginsize{1in}{1in}{1in}{1in}%

\newtheorem{theorem}{Theorem}[section]
\newtheorem{corollary}[theorem]{Corollary}
\newtheorem{cor}[theorem]{Corollary}
\newtheorem{Thm}[theorem]{Theorem}
\newtheorem{assertion}[theorem]{Assertion}
\newtheorem{lem}[theorem]{Lemma}
\newtheorem{lemma}[theorem]{Lemma}
\newtheorem{proposition}[theorem]{Proposition}

\newtheorem{exa}[theorem]{Example}

\newtheorem{definition}[theorem]{Definition}
\newtheorem{Def}[theorem]{Definition}
\newtheorem{ass}[theorem]{Assumption}
\newtheorem{assumption}[theorem]{Assumption}
\newtheorem{assumptions}[theorem]{Assumptions}
\newtheorem{remark}[theorem]{Remark}
\newtheorem{Rem}[theorem]{Remark}

\newtheorem*{claim*}{Claim}
\newtheorem*{notation}{Notation}

\newcommand{\diver}{\mathrm{div}\,}

\newcommand{\R}{\mathbb{R}}

\newcommand{\N}{\mathbb{N}}

\newcommand{\loc}{{\rm loc}}
\numberwithin{equation}{section}

\newcommand{\Hmm}[1]{\leavevmode{\marginpar{\tiny%
			$\hbox to 0mm{\hspace*{-0.5mm}$\leftarrow$\hss}%
			\vcenter{\vrule depth 0.1mm height 0.1mm width \the\marginparwidth}%
			\hbox to
			0mm{\hss$\rightarrow$\hspace*{-0.5mm}}$\\\relax\raggedright #1}}}

\DeclareMathOperator{\diam}{diam}
\DeclareMathOperator{\supp}{supp}
\DeclareMathOperator{\capacity}{Cap}
\newcommand{\core}{C_c^{\infty}(\Omega)}

\newcommand{\dx}{\,\mathrm{d}x}

\DeclareMathOperator{\dive}{div}

\def\<{\langle}
\def\>{\rangle}

\long\def\prob#1\soln#2\endps{{\color{blue}#1}\medskip\par
	\noindent\underline{\sc Solution}:\hspace*{1em}\parindent=2em #2}

\AtBeginDocument{\begin{CJK}{UTF8}{gbsn}}
\AtEndDocument{\end{CJK}}

       \def\gd{\delta}      
                         
\def\gf{\phi}       \def\vgf{\varphi}    
      \def\gk{\kappa}      \def\gl{\lambda}
\def\gm{\mu}

      \def\gw{\omega}

\def\Gw{\Omega}              
\begin{document}
\pagenumbering{gobble}
\title{\textbf{Positive solutions of the~$\mathcal{A}$-Laplace equation with a potential}}
\author{Yongjun Hou\thanks{Guangdong Technion -- Israel Institute of Technology, 241 Daxue Road, Shantou, Guangdong, China, 515063 and Department of Mathematics, Technion -- Israel Institute of Technology, Haifa, Israel, 3200003. {Email: yongjun.hou@campus.technion.ac.il; houmathlaw@outlook.com}},
Yehuda Pinchover\thanks{Department of Mathematics, Technion -- Israel Institute of Technology, Haifa, Israel, 3200003, Email: pincho@technion.ac.il}, and
Antti Rasila\thanks{Corresponding author. Guangdong Technion -- Israel Institute of Technology, 241 Daxue Road, Shantou, Guangdong, China, 515063 and  Department of Mathematics, Technion -- Israel Institute of Technology, Haifa, Israel, 3200003. Email: antti.rasila@iki.fi; antti.rasila@gtiit.edu.cn}}
\date{December 02, 2024}
\maketitle
\pagenumbering{arabic}

\vspace{-10mm}

	\begin{abstract}In this paper, we study positive solutions of the quasilinear elliptic equation
$$Q'_{p,\mathcal{A},V}[u]\triangleq-\mathrm{div}{\mathcal{A}(x,\nabla u)}+V(x)|u|^{p-2}u=0,$$
in a domain $\Omega\subseteq \mathbb{R}^n$, where $n\geq 2$, $1<p<\infty$, the divergence of $\mathcal{A}$ is the well known $\mathcal{A}$-Laplace operator considered in the influential book of  Heinonen, Kilpel\"{a}inen, and Martio, and the potential $V$ belongs to a certain local Morrey space. The main aim of the paper is to extend criticality theory to the operator $Q'_{p,\mathcal{A},V}$. In particular, we prove an Agmon-Allegretto-Piepenbrink (AAP) type theorem, establish the uniqueness and simplicity of the principal eigenvalue of $Q'_{p,\mathcal{A},V}$ in a domain $\omega\Subset\Omega$, and give various  characterizations of criticality. Furthermore, we also study positive solutions of the equation $Q'_{p,\mathcal{A},V}[u]=0$ of minimal growth at infinity in $\Omega$, the existence of a minimal positive Green function, and the minimal decay at infinity of Hardy-weights.
\end{abstract}
\medskip


\section{Introduction}\label{sec1}
Let $\Omega\subseteq \mathbb{R}^{n}~(n\geq 2)$ be a domain and $1<p<\infty$. In the monograph \cite{HKM} by Heinonen, Kilpel\"{a}inen, and  Martio, the authors studied a nonlinear potential theory for the second-order quasilinear elliptic operator~$\dive(\mathcal{A}(x,\nabla u))$, which is called the $\mathcal{A}$-Laplace operator (or in short, the $\mathcal{A}$-Laplacian). We recall that the $\mathcal{A}$-Laplacian might be degenerate or singular elliptic operator that satisfies some natural local regularity assumptions. In addition, it is assumed that the operator $\mathcal{A}$ is $(p-1)$-homogeneous and monotone in its second variable (for details, see Assumption \ref{ass8}).  Prototypes of the $\mathcal{A}$-Laplace operator are the $p$-Laplacian $\dive{\left(|\nabla u|^{p-2}\nabla u\right)}$ and the $(p,A)$-Laplacian
$$\dive{\left(|\nabla u|^{p-2}_{A}A\nabla u\right)}\triangleq
\diver\left((A(x)\nabla u\cdot\nabla u)^{(p-2)/2}A(x)\nabla u\right),$$
where $A$ is a locally bounded, {symmetric,} and locally uniformly positive definite matrix function (see, \cite{HKM,Pinchover,Regev,Tintarev}).

A systematic criticality theory has been developed for the $p$-Laplace operator and the~$(p,A)$-Laplace operator with a locally bounded potential in \cite{Tintarev} and \cite{Regev}, respectively. Furthermore, in \cite{Pinchover}, Pinchover and Psaradakis have extended the theory to the case of the $(p,A)$-Laplace operator with a potential in the local Morrey space. See \cite{Murata, Pinchoverlinear} for the criticality theory for the second-order linear  elliptic (not necessarily symmetric) case. We refer also to Pinsky's book \cite{Pinsky}, where the author studies this topic from the probabilistic point of view. Moreover, a criticality theory for Schr\"{o}dinger operators on graphs has also been established by Keller, Pinchover, and Pogorzelski in \cite{Keller}. {For recent papers on criticality theory for a certain class of quasilinear operators on graphs, see \cite{Fischer4, Fischer7}.} The theory has witnessed its applications in the works of Murata and Pinchover and  their collaborators (see recent examples in \cite{Beckus, KellerHardy, MT}). For the case of generalized Schr\"{o}dinger forms, we refer to \cite{Takeda2014, Takeda2016}.

Criticality theory has applications in a number of areas of analysis, for example, in spectral theory of Schr\"odinger operators \cite{Pinchoverlinear}, variational inequalities (like Hardy, Rellich, and  Hardy-Sobolev-Maz'ya type inequalities) \cite{Kovarik,HSM}, and stochastic processes \cite{Pinsky}.  Among the applications in PDE we mention results concerning the large time behavior of  the heat kernel \cite{PinchoverGreen}, Liouville-type theorems \cite{Lioupincho}, the behavior of the minimal positive Green function \cite{PinchoverGreen2,Pinchoverlinear}, and the asymptotic behavior of positive solutions near an isolated singularity \cite{Fraas}.

The goal of the present paper is to extend the results in \cite{Pinchover,Regev,Tintarev} concerning positive solutions of the homogeneous quasilinear equation$$Q'_{p,A,V}[u]\triangleq -\dive{\left(|\nabla u|^{p-2}_{A(x)}A(x)\nabla u\right)}+V(x)|u|^{p-2}u=0\quad \mbox{ in } \Omega,$$  to the equation
	$$Q'_{p,\mathcal{A},V}[u]\triangleq -\dive{\mathcal{A}(x,\nabla u)}+V(x)|u|^{p-2}u=0\quad \mbox{ in } \Omega.$$

The latter equation is the {\em local} Euler-Lagrange equation of the energy functional
$$Q_{p,\mathcal{A},V}[\vgf]\triangleq Q_{p,\mathcal{A},V}[\vgf;\Omega]\triangleq \int_{\Omega}\left(\mathcal{A}(x,\nabla \vgf)\cdot\nabla \vgf + V(x)|\vgf|^{p}\right)\dx \qquad \vgf\in\core.$$
Note that the equation $Q'_{p,\mathcal{A},V}[u]=0$ (and in particular, $Q'_{p,A,V}[u]=0$) is {\em half-linear}, that is, if $v$ is a solution of this equation, then for every $c\in\R$, $cv$ is also a solution.
 We assume that the potential $V$ belongs to
 the local Morrey space $M^{q}_{\loc}(p;\Omega)$ associated with the exponent $p$ (see Definitions~\ref{Morreydef1} and \ref{Morreydef2}), which is almost the largest class of potentials that  guarantees the validity of the Harnack inequality and the H\"older continuity of solutions (see also \cite[p. 1319]{Pinchover}).  The assumptions on $\mathcal{A}$ are as in \cite{HKM} (see Assumption \ref{ass8}). In addition, {two  further assumptions (assumptions \ref{ass2} and \ref{ngradb})
are} assumed to prove certain important results in Sections~\ref{sec_eigenvalue}, \ref{criticality}, and \ref{minimal}. In fact, {Assumption \ref{ass2}} is utilized in two different ways. One is direct (see Proposition \ref{mainlemma}). The other is indirect, i.e., via the D\'{\i}az-Sa\'{a}-type inequality Lemma~\ref{elementary} (see Theorem \ref{maximum}).

 Our main results include the existence, uniqueness, and simplicity of the principal eigenvalue of the operator~$Q'_{p,\mathcal{A},V}$ in a domain~$\omega\Subset\Omega$, a weak comparison principle,  and the criticality theory for $Q'_{p,\mathcal{A},V}$. Moreover, based on a  Picone-type identity and a generalized H\"older inequality (see Lemma~\ref{ass1}), two alternative proofs of Agmon-Allegretto-Piepenbrink type (AAP) theorem are given (see Lemma \ref{lem_alter} and Theorem \ref{thm_AAP}, see also \cite{Agmon, Allegretto1974}, and also \cite{Pinchover} for a short updated review on the AAP theorem). In addition, we characterize in a Lipschitz domain~$\omega\Subset\Omega$ the validity of the generalized strong/weak maximum principles and the unique solvability in $W^{1,p}_0(\gw)$ of a nonnegative solution of the Dirichlet problem $Q'_{p,\mathcal{A},V}[v]=g\geq 0$ with~$g\in L^{p'}(\omega)$ via the strict positivity of the principal eigenvalue.

 The paper is organized as follows. In Section \ref{back},
 we introduce a variational Lagrangian $F$ and then obtain from $F$ the operator~$\mathcal{A}$ by virtue of \cite[Lemma 5.9]{HKM}. We establish a generalized H\"{o}lder inequality (Lemma~\ref{ass1}) which is a key result used to prove several fundamental results, and formulate the additional assumption discussed above (Assumption~\ref{ass2}). In addition, we recall the definition of the associated local Morrey spaces, and the Morrey-Adams theorem, which is an essential tool for our study. Finally, we define the notion of weak solutions of the quasilinear equation $Q'_{p,\mathcal{A},V}[u]=0$.
In Section \ref{toolbox}, we present certain a priori properties of weak solutions of the quasilinear equation $Q'_{p,\mathcal{A},V}[u]=0$, including Harnack-type inequalities, local H\"{o}lder estimate, and the Harnack convergence principle.

 In Section \ref{sec_eigenvalue}, we first extend D\'{\i}az-Sa\'{a} type inequalities, and then prove the coercivity and weak lower semicontinuity of certain related functionals. We also establish a Picone-type identity. Then we show that in a domain~$\omega\Subset\Omega$, the generalized principal eigenvalue is a principal eigenvalue, that is a Dirichlet eigenvalue with a nonnegative eigenfunction. Moreover, we prove that the generalized principal eigenvalue is simple. With these preliminaries, we also study the generalized weak and strong maximum principles, the positivity of the generalized principal eigenvalue and other related properties. Furthermore, we establish a weak comparison principle by virtue of the super/sub-solution technique.

In Section \ref{AP}, we prove for our setting the corresponding AAP type theorem which turns out to be closely related to the existence of solutions of a certain nonlinear first-order equation of the divergence type. As a result, we show that the AAP  theorem implies the uniqueness of the principal eigenvalue in a domain~$\omega\Subset\Omega$.

In Section \ref{criticality}, we establish a systematic criticality theory for the operator $Q'_{p,\mathcal{A},V}$ with applications to a Hardy-Sobolev-Maz'ya inequality and the $(\mathcal{A},V)$-capacity.

In Section \ref{minimal}, we study the removability of an isolated singularity. We also show that the criticality of~$Q_{p,\mathcal{A},V}$ is equivalent to the existence of a global minimal positive solution. Moreover, we prove that the existence of a minimal positive Green function, with an additional assumption in the case of~$p>n$, implies the subcriticality  of $Q_{p,\mathcal{A},V}$. Finally, we extend the results in \cite{Kovarik}  and answer the question: How large can Hardy-weights be?
\section{$\mathcal{A}$-Laplacian  and Morrey potentials}\label{back}
In this section, we introduce the~$\mathcal{A}$-Laplace operator.  We recall the local Morrey space where our potential~$V$ lies and the Morrey-Adams theorem, which are defined and proved in \cite{Pinchover}. Finally, we define weak solutions and supersolutions of the quasilinear elliptic equation~$Q'_{p,\mathcal{A},V}[v]=g$.

Let $g_1,g_2$ be two positive functions defined in $\Gw$. We use the notation $g_1\asymp g_2$ in
$\Gw$ if there exists a positive constant $C$ such that $C^{-1}g_{2}(x)\leq g_{1}(x) \leq Cg_{2}(x)$ for all  $x\in \Gw$.
\subsection{Variational Lagrangian $F$ and its gradient $\mathcal{A}$}
In this subsection, we present a variational Lagrangian $F$ which satisfies certain desired conditions. Then we  define the $\mathcal{A}$-Laplacian as the divergence of the gradient of $F$.
\subsubsection{Variational Lagrangian $F$}
Following the assumptions in {\cite[page 97]{HKM}}, we list below our structural and regularity assumptions on the variational Lagrangian $F.$
	\begin{assumptions}\label{ass9}
		{\em
		\label{assump1}
		Let~$\Omega\! \subseteq \!\R^{n}$ be a nonempty domain, let $F:\Gw\times \mathbb{R}^{n} \! \rightarrow \!\mathbb{R}_+$, and let $1\!<\!p\!<\!\infty$. We assume that $F$ satisfies the following conditions:
		\begin{itemize}
  \item {\bf Measurability:} For all~$\xi\in\mathbb{R}^{n}$, the mapping $x\mapsto F(x,\xi)$ is measurable in $\Gw$.
\item {\bf Ellipticity:} For  all $\omega\Subset \Gw$ there exist $0<\kappa_\omega\leq\nu_\omega<\infty$ such that for almost all $x\in \omega$ and all $\xi\in \R^n$,
$\kappa_\omega|\xi|^{p}\leq F(x,\xi)\leq\nu_\omega|\xi|^{p}$.
\item {\bf Convexity and differentiability with respect to $\xi$}: For a.e.~$x\in \Gw$, the mapping $\xi\mapsto F(x,\xi)$ is strictly convex and continuously differentiable in $\R^n$.
\item {\bf Homogeneity:} $F(x,\lambda\xi)=|\lambda|^{p}F(x,\xi)$ for a.e.~$x\in \Gw$, all~$\lambda\in\mathbb{R}$, and all~$\xi\in\mathbb{R}^{n}$.
\end{itemize}		
}
	\end{assumptions}
The following is a useful inequality derived directly from the strict convexity of $F$.
\begin{lemma}[{\cite[Lemma 5.6]{HKM}}]\label{strictconvexity}
For a.e.~$x\in\Omega$ and all~$\xi_{1},\xi_{2}\in\R^{n}$ with~$\xi_{1}\neq\xi_{2}$, we have:
  $$F(x,\xi_{1})-F(x,\xi_{2})>\nabla_{\xi}F(x,\xi_{2})\cdot(\xi_{1}-\xi_{2}).$$
\end{lemma}
\subsubsection{$\mathcal{A}$-Laplacian}
\begin{Def}
  {\em
 Let~$\Omega\subseteq \R^{n}$ be a nonempty domain and $F(x,\xi)$ satisfy Assumptions \ref{ass9}. For a.e.~$x\in \Gw$, we denote by $\mathcal{A}(x,\xi) \triangleq  \nabla_{\xi}F(x,\xi)$ the classical gradient  of $F(x,\xi)$ with respect to~$\xi$. The\emph{~$\mathcal{A}$-Laplacian} is defined as the divergence of~$\mathcal{A}$.
 	}
\end{Def}
 \begin{remark}
 \emph{ By Euler's homogeneous function theorem, for a.e.~$x\in\omega$,
  	 $$\mathcal{A}(x,\xi)\cdot\xi =
  	p F(x,\xi)  \geq p\gk_\gw |\xi|^p \qquad \forall \xi\in \R^n .$$
 Moreover,  since for a.e. $x\in \Gw$ the nonnegative function
 \begin{equation}\label{newformula}
 \vert\xi\vert_{\mathcal{A}}=\vert\xi\vert_{\mathcal{A}(x)}\triangleq  (\mathcal{A}(x,\xi)\cdot\xi)^{1/p}
 \end{equation}
  is positively homogeneous of degree $1$ in $\xi$, and $\{\xi \in \R^n \mid \vert\xi\vert_{\mathcal{A}}\leq 1\}$ is a convex set, it follows that for a.e. $x\in \Gw$,
  $\vert\xi\vert_{\mathcal{A}}$
 is  a norm on $\R^n$  (see, for example, \cite[Theorem 1.9]{Simon}).}
 \end{remark}
\begin{Thm}[{\cite[Lemma 5.9]{HKM}}]\label{thm_1}
Let~$\Omega\subseteq \R^{n}$ be a nonempty domain. For every domain $\omega\Subset\Omega$, denote  $\alpha_{\omega}=\kappa_{\omega}$, $\beta_{\omega}=2^{p}\nu_{\omega}$. Then the vector-valued function~$\mathcal{A}(x,\xi):  \Gw\times \mathbb{R}^{n}\rightarrow \mathbb{R}^{n}$ satisfies the following conditions:
  \begin{itemize}
  \item {\bf Regularity:} For a.e. $x\in \Gw$, the function
  $\mathcal{A}(x,\xi ): \mathbb{R}^{n} \rightarrow \mathbb{R}^{n}$
  is continuous with respect to $\xi$, and  $x \mapsto \mathcal{A}(x,\xi)$ is Lebesgue measurable in $\Gw$ for all~$\xi\in \mathbb{R}^{n}$.
  \item {\bf Homogeneity:} For all~$\lambda\in {\mathbb{R}\setminus\{0\}}$,
  $\mathcal{A}(x,\lambda \xi)=\lambda\,|\lambda|^{p-2}\,\mathcal{A}(x,\xi).$
  \item {\bf Ellipticity:} For all domains $\omega\Subset \Gw$, all $\xi \in \mathbb{R}^{n}$, and a.e. $x\in \omega$,
\begin{equation}\label{structure}
\alpha_\omega|\xi|^{p}\le\mathcal{A}(x,\xi)\cdot\xi,
\quad
|\mathcal{A}(x,\xi)|\le \beta_\omega\,|\xi|^{{p}-1}.
\end{equation}
\item {\bf Monotonicity:} For a.e.~$x\!\in\! \Gw$ and all~$\xi\!\neq \! \eta \! \in \! \mathbb{R}^{n}$,
$\big(\mathcal{A}(x,\xi)-\mathcal{A}(x,\eta)\big) \! \cdot \! (\xi-\eta)> 0.$
\end{itemize}
\end{Thm}
	\begin{ass}\label{ass8}
		{\em Throughout the paper we assume that $\mathcal{A}(x,\xi)=\nabla_{\xi}F(x,\xi)$, where $F$ satisfies Assumptions~\ref{ass9}. In particular, we assume that  $\mathcal{A}$ satisfies all the conditions mentioned in Theorem~\ref{thm_1}.
		}
	\end{ass}
 \subsubsection{Generalized H\"older inequality}
In the proof of the AAP Theorem (Theorem~\ref{thm_AAP}), 
we use the following generalized H\"older inequality. The inequality follows similarly to the proof of \cite[Lemma 2.2]{newpicone}, where the case $\mathcal{A}=\mathcal{A}(\xi)$ is considered. Nevertheless, since the generalized H\"older inequality is a pointwise  inequality with respect to $x$, the proof holds also for $\mathcal{A}=\mathcal{A}(x,\xi)$.
\begin{lemma}[Generalized H\"older inequality]\label{ass1}
 Let~$p'$ be the conjugate exponent of~$1<p<\infty$. Then the following inequality holds $$\big|\mathcal{A}(x,\xi)\cdot\eta\big|
 	\leq\big(\mathcal{A}(x,\xi)\cdot\xi\big)^{1/p'}\big(\mathcal{A}(x,\eta)\cdot\eta\big)^{1/p}
    =\vert\xi\vert_{\mathcal{A}}^{p-1}\vert\eta\vert_{\mathcal{A}}, \qquad \forall \xi,\eta\in\mathbb{R}^{n} \mbox{ and a.e. } x\in \Gw.$$
\end{lemma}
\subsubsection{{Stronger convexity assumption on $|\xi|_{\mathcal{A}}^p$}}
By our assumptions, for a.e. $x\in \Gw$, the function  $\xi\mapsto |\xi|_{\mathcal{A}}^p$ defined by \eqref{newformula} is strictly convex. For certain important results in Sections~\ref{sec_eigenvalue}, \ref{criticality}, and \ref{minimal} we need to assume:
\begin{ass}[{Stronger convexity assumption}]
\label{ass2}
  {\em
{We suppose that for every subdomain $\gw\Subset \Gw$ there exists
	a positive constant $C_\gw(p, \mathcal{A})$   such that
$$|\xi|^{p}_{\mathcal{A}}-|\eta|^{p}_{\mathcal{A}}-p\mathcal{A}(x,\eta)\cdot(\xi-\eta)\geq C_\gw(p, \mathcal{A}) [\xi,\eta]_{\mathcal{A}}\qquad \forall \xi,\eta\in \R^n \mbox{ and a.e. } x\in \gw,$$
where
\begin{equation}\label{eq_strict}
[\xi,\eta]_{\mathcal{A}}\triangleq
 \begin{cases}
   |\xi-\eta|^{p}_\mathcal{A}&\mbox{if $p\geq 2$,}\\
(|\eta|_{\mathcal{A}}+|\xi-\eta|_{\mathcal{A}})^{p-2}|\xi-\eta|^{2}_\mathcal{A}&\mbox{if $1<p<2$.}
   \end{cases}
\end{equation}
}}
\end{ass}
\begin{Rem}\label{pAlaplacian}
  {\em See \cite[Lemma 2.2]{Lioupincho} and \cite[Lemma 3.4]{Regev}  for such inequalities for the~$p$-Laplacian and the~$(p,A)$-Laplacian. {Note that in these two papers, double-sided inequalities called the {\em simplified energy}  are proved. Also, Assumption \ref{ass2} clearly implies strict convexity.}}
\end{Rem}
\subsubsection{Pseudo $p$-Laplacian}
We present further examples of operators which fulfill the assumptions above.
\begin{Def}
	\emph{A measurable matrix function $A:\Gw\to \R^{n^2}$ is called \emph{locally bounded} if for every subdomain $\omega\Subset\Omega$, there exists a positive constant~$C(\omega)$ such that,~$|A(x)\xi|\leq C(\omega)|\xi|$ for all  $\xi\in\R^n$ and  a.e. $x\in \omega$.}
\end{Def}
\begin{exa}\label{exa}
\emph{For a.e.~$x\in\Omega$ and every~$\xi =(\xi_1,\ldots,\xi_n)\in\mathbb{R}^{n}$, let	$$F(x,\xi)\triangleq \frac{1}{p}\sum_{i=1}^{n}a_{i}(x)|\xi_{i}|^{p},$$  	where {$1<p<\infty$} and the Lebesgue measurable functions locally satisfy~$a_{i}\asymp 1$.}
\end{exa}
\begin{lemma}\label{pseudo}
 Let~$F$ be as in Example \ref{exa}. For a.e.~$x\in\Omega$ and every~$\xi =(\xi_1,\ldots,\xi_n)\in\mathbb{R}^{n}$, we have
  \begin{enumerate}
   \item[$(1)$] \!$\nabla_{\xi}F(x,\xi)\!=\!\mathcal{A}(x,\xi)\!=\!(a_{1}(x)|\xi_{1}|^{p-2}\xi_{1},\ldots,a_{n}(x)|\xi_{n}|^{p-2}\xi_{n})$, \mbox{and}
$|\xi|_{\mathcal{A}}^{p}\!=\!\sum_{i=1}^{n}a_{i}(x)|\xi_{i}|^{p}$;
   \item[$(2)$] the operator $\mathcal{A}$ satisfies assumptions~\ref{ass8} and \ref{ass2}.
\end{enumerate}
 Furthermore,
 \begin{enumerate} \item[$(3)$]for $0\leq t\leq 1$ {and~$p\geq 2$},  consider the Lagrangian $F_{t,A}\triangleq tF+((1-t)/p)|\xi|_{A}^{p}$, where $A$ is a locally bounded, symmetric, and locally uniformly positive definite matrix function. Then $\mathcal{A}_{t,A}$, the gradient  of $F_{t,A}$,  satisfies assumptions~\ref{ass8} and \ref{ass2}.
 \end{enumerate}
 \end{lemma}
\begin{remark}
\emph{If $a_{i}= 1$ for all~$i=1,2,\ldots,n$, then the operator $\diver\!(\mathcal{A})$ is called the {\em pseudo $p$-Laplacian}.}
\end{remark}
\begin{proof}[Proof of Lemma \ref{pseudo}]
Part (1) is obtained by a straightforward differentiation.

(2) {Let~$p\geq 2$ (for~$p<2$, see \cite[Remark 8.27]{Hou}).} Our proof is inspired by \cite[Lemma 4.2]{Lindqvist}. Since $\displaystyle{\sum_{i=1}^{n}a_{i}(x)|\xi_{i}|^{p}}$ is convex with respect to $\xi$, {we get} $$\sum_{i=1}^{n}a_{i}(x)|\xi_{i}|^{p}\geq \sum_{i=1}^{n}a_{i}(x)|\eta_{i}|^{p}+p\sum_{i=1}^{n}a_{i}(x)|\eta_{i}|^{p-2}\eta_{i}(\xi_{i}-\eta_{i}),$$
  for a.e.~$x\in\Omega$ and all~$\xi,\eta\in\mathbb{R}^{n}$.
  Hence,
  $$\sum_{i=1}^{n}a_{i}(x)\left\vert\frac{\xi_{i}+\eta_{i}}{2}\right\vert^{p}\geq \sum_{i=1}^{n}a_{i}(x)|\eta_{i}|^{p}+\frac{p}{2}\sum_{i=1}^{n}a_{i}(x)|\eta_{i}|^{p-2}\eta_{i}(\xi_{i}-\eta_{i}).$$
  By Clarkson's inequality for $p\geq 2$ \cite[Theorem 4.10]{Brezis}, we obtain
  $$\sum_{i=1}^{n}a_{i}(x)|\xi_{i}|^{p}+\sum_{i=1}^{n}a_{i}(x)|\eta_{i}|^{p}\geq 2\sum_{i=1}^{n}a_{i}(x)\left\vert\frac{\xi_{i}+\eta_{i}}{2}\right\vert^{p} +2\sum_{i=1}^{n}a_{i}(x)\left\vert\frac{\xi_{i}-\eta_{i}}{2}\right\vert^{p}.$$
  Then
  $$\sum_{i=1}^{n}a_{i}(x)|\xi_{i}|^{p}\geq \sum_{i=1}^{n}a_{i}(x)|\eta_{i}|^{p}+ p\sum_{i=1}^{n}a_{i}(x)|\eta_{i}|^{p-2}\eta_{i}(\xi_{i}-\eta_{i}) +2^{1-p}\sum_{i=1}^{n}a_{i}(x)\left\vert\xi_{i}-\eta_{i}\right\vert^{p},$$ which gives Assumption \ref{ass2} for $p\geq 2$ because locally~$a_{i}\asymp 1$ for all~$i=1,2,\ldots,n$.

  Moreover, on the unit Euclidean sphere in~$\mathbb{R}^{n}$, the function~$f(\xi)\triangleq\sum_{i=1}^{n}a_{i}(x)|\xi_{i}|^{p}$ has a positive lower bound and a finite upper bound, and therefore, the local ellipticity conditions follow.

  (3) For all~$\xi\in\R^{n}$ and a.e.~$x\in\Omega$, $|\xi|_{\mathcal{A}_{t,A}}^{p}= t|\xi|_{\mathcal{A}}^{p}+(1-t)|\xi|_{A}^{p}$. Hence,
  {Assumption~\ref{ass2}} and the ellipticity conditions of~$|\cdot|_{\mathcal{A}_{t,A}}^{p}$ follow from Remark \ref{pAlaplacian} and~$(2)$.
\end{proof}
\subsection{Morrey potentials}
In this subsection, we give a short review of the local Morrey space $M^{q}_{\mathrm{loc}}(p;\Omega)$, the functional space where the potential~$V$ belongs to, and recall the Morrey-Adams theorem.
\subsubsection{Local Morrey space $M^{q}_{\mathrm{loc}}(p;\Omega)$}
The following is a reformulation of the local Morrey space $M^{q}_{\mathrm{loc}}(p;\Omega)$, where $q=q(p)$.
  \begin{Def}[{\cite[definitions 2.1 and 2.3]{Pinchover}}]\label{Morreydef1}{\em
  Let $\omega\Subset \Omega$ be a domain and $f\in L^1_\loc(\omega)$ be a real-valued function. Then
  \begin{itemize}
  \item for $p<n$, we say that $f\in M^{q}(p;\omega)$ if $q>n/p$ and
  $$\Vert f\Vert_{M^{q}(p;\omega)}\triangleq \sup_{\substack{y\in\gw\\0<r<\diam(\gw)}}
  \frac{1}{r^{n/q'}}\int_{\omega\cap B_{r}(y)}|f|\dx<\infty,$$ where $\mathrm{diam}(\omega)$ is the diameter of~$\omega$; 
  \item for $p=n$, we say that~$f\in M^{q}(n;\omega)$ if  $q>n$ and
$$\Vert f\Vert_{M^{q}(n;\omega)}\triangleq \sup_{\substack{y\in\gw\\0<r<\diam(\gw)}} \varphi_{q}(r)\int_{\omega\cap B_{r}(y)}|f|\dx<\infty,$$
where $\varphi_{q}(r)\triangleq \Big(\log\big(\mathrm{diam}(\omega)/r\big)\Big)^{q/n'};$
\item for $p>n$ and $q=1$, we define~$M^{q}(p;\omega)\triangleq L^{1}(\omega)$.
\end{itemize}
}
\end{Def}
\begin{Def}[{\cite[Definition 2.3]{Pinchover}}]\label{Morreydef2}{\em
For every real-valued function $f\in L^1_\loc(\Omega)$ and~$1<p<\infty$, we say that $f\in M^{q}_{\loc}(p;\Omega)$ if $f\in M^{q}(p;\omega)$ for every domain~$\omega\Subset\Omega$.
}
\end{Def}
For a more detailed discussion on Morrey spaces, see \cite{Maly,Pinchover} and references therein.
\subsubsection{Morrey-Adams theorem}
We present the Morrey-Adams theorem proved in \cite{Pinchover}, which is crucial when dealing with the potential term. See \cite{Maly,Morrey1966,Rakotoson1990,Trudinger1967} for relevant earlier results.
\begin{Thm}[{\cite[Theorem 2.4]{Pinchover}}]\label{MA_thm}
	Let~$\omega\Subset\mathbb{R}^{n}$ be a domain and~$V\in M^{q}(p;\omega)$.
	\begin{enumerate}
	\item[$(1)$] There exists a constant~$C(n,p,q)>0$ such that for any~$\delta>0$ and all~$u\in W^{1,p}_{0}(\omega)$,
	\begin{equation*}
	\int_{\omega}|V||u|^{p}\dx\leq \delta\Vert\nabla u\Vert^{p}_{L^{p}(\omega;\mathbb{R}^{n})}+\frac{C(n,p,q)}{\delta^{n/(pq-n)}}\Vert V\Vert^{pq/(pq-n)}_{M^{q}(p;\omega)}\Vert u\Vert^{p}_{L^{p}(\omega)}.
	\end{equation*}
	\item[$(2)$] For any~$\omega'\Subset\omega$ with Lipschitz boundary, there exists $\delta_{0}$ such that for any~$0<\delta\leq \delta_{0}$ and all~$u\in W^{1,p}(\omega')$,
	\begin{equation*}
	\int_{\omega'}|V||u|^{p}\dx\leq \delta\Vert\nabla u\Vert^{p}_{L^{p}(\omega';\mathbb{R}^{n})}+C\left(n,p,q,\omega',\omega,\delta,\Vert V\Vert_{M^{q}(p;\omega)}\right)\Vert u\Vert^{p}_{L^{p}(\omega')}.
	\end{equation*}
\end{enumerate}
\end{Thm}
\subsection{Weak solutions of $Q'_{p,\mathcal{A},V}[u]=g$}
With the preliminaries of the previous subsections in hand, we may define weak solutions of the equation~$Q'_{p,\mathcal{A},V}[u]=g$.
\begin{Def}\label{def_sol}
{\em Suppose that $\mathcal{A}$ satisfies Assumption~\ref{ass8} and~$V, g\in M^{q}_{\mathrm{loc}}(p;\Omega)$. A function~$v\in W^{1,p}_{\loc}(\Omega)$ is a {\em (weak) solution} of the equation
  \begin{equation}\label{half}
  Q'_{p,\mathcal{A},V}[v]\triangleq -\dive\mathcal{A}(x,\nabla v)+V|v|^{p-2}v=g,
\end{equation}
in~$\Omega$ if for all~$\vgf \in C_{c}^{\infty}(\Omega)$,$$\int_{\Omega}\mathcal{A}(x,\nabla v)\cdot \nabla \vgf\dx+\int_{\Omega}V|v|^{p-2}v \vgf\dx=\int_{\Omega} g\vgf\dx,$$ it is a \emph{supersolution} of \eqref{half} if for all nonnegative~$\vgf \in C_{c}^{\infty}(\Omega)$,$$\int_{\Omega}\mathcal{A}(x,\nabla v)\cdot \nabla \vgf\dx+\int_{\Omega}V|v|^{p-2}v \vgf\dx\geq \int_{\Omega} g\vgf\dx,$$
 and it is a \emph{subsolution} of \eqref{half} if for all nonnegative~$\vgf \in C_{c}^{\infty}(\Omega)$,$$\int_{\Omega}\mathcal{A}(x,\nabla v)\cdot \nabla \vgf\dx+\int_{\Omega}V|v|^{p-2}v \vgf\dx\leq \int_{\Omega} g\vgf\dx.$$
A supersolution~$v\in W^{1,p}_{\loc}(\Omega)$ of \eqref{half} is said to be \emph{proper} if~$v$ is not a solution of \eqref{half}.
}
\end{Def}
By virtue of the Morrey-Adams theorem and an approximation argument, we obtain:
\begin{lemma}\label{lem4.11}
  Suppose that $\mathcal{A}$ satisfies Assumption~\ref{ass8} and~$V\in M^{q}_{\mathrm{loc}}(p;\Omega)$.
  \begin{enumerate}
  \item[$(1)$] All the integrals in Definition \ref{def_sol} are well defined.
  \item[$(2)$] The test function space~$C_{c}^{\infty}(\Omega)$ in Definition \ref{def_sol} can be replaced  with~$W^{1,p}_{c}(\Omega)$.
  \end{enumerate}
\end{lemma}
\section{Properties of weak solutions of~$Q'_{p,\mathcal{A},V}[u]=0$}\label{toolbox}
In this section, we present various properties of weak solutions of~$Q'_{p,\mathcal{A},V}[u]=0$ which are frequently used subsequently, including Harnack and weak Harnack inequalities, standard elliptic H\"{o}lder estimates, and a Harnack convergence principle.
\subsection{Harnack inequality}
By \cite[Theorem 3.14]{Maly} for~$p\!\leq \!n$ and \cite[Theorem 7.4.1]{Pucci} for~$p\!>\!n$, we have the following local Harnack inequality for nonnegative solutions of $Q'_{p,\mathcal{A},V}[u]\!=\!0$. See \cite{Trudinger,Maly,Moser,Serrin1964} for Harnack's inequalities for linear and quasilinear equations in divergence form.
\begin{Thm}
	Assume that~$\mathcal{A}$ satisfies Assumption \ref{ass8} and~$V\in M^{q}_{\loc}(p;\Omega)$. Let~$v$ be a nonnegative solution~$v$ of $Q'_{p,\mathcal{A},V}[u]=0$ in a domain $\gw\Subset \Omega$. Then for any~$\omega'\Subset\omega$,
	$$\sup_{\omega'}v\leq C\inf_{\omega'} v,$$ where~$C$ is a positive constant depending only on~$n,p,q,\omega',\omega,\alpha_{\omega},\beta_{\omega},$ and $\Vert V\Vert_{M^{q}(p;\omega)}$.
\end{Thm}
\subsection{H\"{o}lder estimate}
Let $v$ be a  H\"{o}lder continuous function of the order $0<\gamma\leq 1$ in $\gw$. We denote
$$[v]_{\gamma,\omega}\triangleq\sup_{x,y\in\omega,x\neq y}\frac{\big|v(x)-v(y)\big|}{|x-y|^{\gamma}}\,.$$
The H\"{o}lder continuity of solutions of $Q'_{p,\mathcal{A},V}[u]=0$ follows from \cite[Theorem 4.11]{Maly} for~$p\leq n$ and \cite[Theorem 7.4.1]{Pucci} for~$p>n$. For further regularity of solutions of quasilinear elliptic equations, see \cite{Trudinger, Maly, Pucci}. We need the following result:
\begin{Thm}
	Assume that~$\mathcal{A}$ satisfies Assumption \ref{ass8} and~$V\in M^{q}_{\loc}(p;\Omega)$. Let~$v$ be a solution of $Q'_{p,\mathcal{A},V}[u]=0$ in a domain   $\gw\Subset \Omega$. Then~$v$ is locally H\"{o}lder continuous of order $0<\gamma\leq 1$ (depending on~$n,p,q,\alpha_{\omega}$, and~$\beta_{\omega}$). Furthermore, for any~$\omega'\Subset\omega$,$$[v]_{\gamma,\omega'}\leq C\sup_{\omega}|v|,$$ where~$C$ is a positive constant depending only on~$n,p,q,\omega',\omega,\alpha_{\omega},\beta_{\omega}$, and~$\Vert V\Vert_{M^{q}(p;\omega)}$.
\end{Thm}
  \subsection{Weak Harnack inequality}
 If $v$ is a nonnegative supersolution of \eqref{half}, then the Harnack inequality still holds for $p>n$ by \cite[Theorem 7.4.1]{Pucci} (See also \cite{Trudinger1967}).
  On the other hand, for~$p\leq n$, we have:
  \begin{Thm}[{\cite[Theorem 3.13]{Maly}}]
     Assume that~$\mathcal{A}$ satisfies Assumption \ref{ass8} and~$V\in M^{q}_{\loc}(p;\Omega)$. Let~$p\leq n$ and~$s=n(p-1)/(n-p)$. For any nonnegative supersolution~$v$ of   $Q'_{p,\mathcal{A},V}[u]=0$ in a domain  $\omega\Subset \Omega$, any~$\omega'\Subset\omega$, and any~$0<t<s$,
    $$\Vert v\Vert_{L^{t}(\omega')}\leq C\inf_{\omega'}v,$$ where~$C$ is a positive constant depending only on~$n,p,t,\omega,\omega',$ and~$\Vert V\Vert_{M^{q}(p;\omega)}$. In particular, such a supersolution is either strictly positive in the domain $\gw$ or vanishes identically.
  \end{Thm}
\subsection{Harnack convergence principle}
In this subsection, we generalize the Harnack convergence principle \cite[Proposition 2.11]{Pinchover} to our setting. See \cite[Proposition 2.7]{Giri} for a slightly more general Harnack convergence principle in the sense that the second-order term is also not fixed but a sequence.
\begin{Def}
\emph{  By a \emph{Lipschitz exhaustion} of~$\Omega$, we mean a sequence of Lipschitz domains~$\{\omega_{i}\}_{i\in\mathbb{N}}$ satisfying for all~$i\in\mathbb{N}$,~$\omega_{i}\Subset\omega_{i+1}\Subset\Omega$ and~$\cup_{i=1}^{\infty}\omega_{i}=\Omega.$}
\end{Def}
For the existence of a Lipschitz exhaustion of~$\Omega$, see  for example \cite[Proposition 8.2.1]{smooth}.
\begin{Thm}[Harnack convergence principle]\label{HCP}
 Let~$\mathcal{A}$ satisfy Assumption~\ref{ass8} and let $\{\omega_{i}\}_{i\in\mathbb{N}}$ be a Lipschitz exhaustion of~$\Omega$ and~$x_{0}\in \omega_{1}$. Assume that~$\mathcal{V}_{i}\in M^{q}(p;\omega_{i})$ converges weakly in~$M^{q}_{\loc}(p;\Omega)$ to~$\mathcal{V}\in M^{q}_{\loc}(p;\Omega)$ as~$i\rightarrow\infty$. For every~$i\in\mathbb{N}$, suppose that~$v_{i}$ is a positive solution of the equation~$Q'_{p,\mathcal{A},\mathcal{V}_{i}}[u]=0$ in~$\omega_{i}$ with~$v_{i}(x_{0})=1$. Then there exists a subsequence of~$\{v_{i}\}_{i\in\mathbb{N}}$ converging weakly in $W^{1,p}_{\loc}(\Omega)$ and locally uniformly in~$\Omega$ to a positive weak solution~$v\in W^{1,p}_{\loc}(\Omega)$ of the equation~$Q'_{p,\mathcal{A},\mathcal{V}}[u]=0$ in~$\Omega$.
\end{Thm}
\begin{proof}
 We use the same approach as in the proof of \cite[Proposition 2.11]{Pinchover}.
Note our convention throughout the proof: when extracting a suitable subsequence of~$\{v_{i}\}_{i\in\mathbb{N}}$, we keep denoting the obtained subsequence by~$\{v_{i}\}_{i\in\mathbb{N}}$ without stating it.

 By the local H\"{o}lder continuity,~$v_{i}$ are continuous in~$\omega_{i}$ for all~$i\in\mathbb{N}$. Fix  a subdomain  $\gw_1\Subset \gw\Subset \Gw$. By the local Harnack inequality, $\{v_{i}\}_{i\in\mathbb{N}}$ is uniformly bounded in $\omega$. Therefore, the local H\"{o}lder continuity guarantees that $\{v_{i}\}_{i\in\mathbb{N}}$ is equicontinuous over~$\omega$. Applying the Arzel\`{a}-Ascoli theorem, we obtain a subsequence converging uniformly in $\gw$ to a positive continuous function $v$.

 Now we aim to find a subsequence of~$\{v_{i}\}_{i\in\mathbb{N}}$ converging weakly in~$W^{1,p}(\omega)$ to a positive solution of~$Q'_{p,\mathcal{A},\mathcal{V}}[u]=0$ in~$\Omega$.  Fix $k\in\mathbb{N}$. Then for any~$\varphi\in C^{\infty}_{c}(\omega_{k})$,  we have~$v_{i}|\varphi|^{p}\in W^{1,p}_{c}(\omega_{k})$ for $i>k$. Testing~$v_{i}|\varphi|^{p}$ in the definition of~$v_{i}$ being a positive weak solution of the equation~$Q'_{p,\mathcal{A},\mathcal{V}_{i}}[u]=0$ in $\gw_k$, we obtain
 $$\left\Vert|\nabla v_{i}|_{\mathcal{A}}\varphi\right\Vert^{p}_{L^{p}(\omega_{k})}\leq p\int_{\omega_{k}}|\nabla v_{i}|_{\mathcal{A}}^{p-1}|\varphi|^{p-1}v_{i}|\nabla \varphi|_{\mathcal{A}}\dx+\int_{\omega_{k}}|\mathcal{V}_{i}|v_{i}^{p}|\varphi|^{p}\dx.$$
 Applying Young inequality $pab\leq \varepsilon a^{p'}+\left((p-1)/\varepsilon\right)^{p-1}b^{\,p},$ on $p\!\int_{\omega_{k}}\!|\nabla v_{i}|_{\mathcal{A}}^{p-1}|\varphi|^{p-1}v_{i}|\nabla \varphi|_{\mathcal{A}}\!\dx$ with $\varepsilon\in (0,1), a=|\nabla v_{i}|_{\mathcal{A}}^{p-1}|\varphi|^{p-1}$, and~$b=v_{i}|\nabla \varphi|_{\mathcal{A}}$, and the Morrey-Adams theorem (Theorem~\ref{MA_thm})  on $\int_{\omega_{k}}|\mathcal{V}_{i}|v_{i}^{p}|\varphi|^{p}\dx$, we conclude:
  \begin{equation*}
    (1-\varepsilon)\!\left\Vert|\nabla v_{i}|_{\mathcal{A}}\varphi\right\Vert^{p}_{L^{p}(\omega_{k})}
    \leq  \!\left(\frac{p-1}{\varepsilon}\right)^{p-1}\!\!\left\Vert|\nabla \varphi|_{\mathcal{A}}v_{i}\right\Vert^{p}_{L^{p}(\omega_{k})}+\delta\left\Vert\nabla(v_{i}\varphi)\right\Vert^{p}_{L^{p}(\omega_{k};\mathbb{R}^{n})}
    +C\left\Vert v_{i}\varphi\right\Vert^{p}_{L^{p}(\omega_{k})},
  \end{equation*}
 where $C=C\left(n,p,q,\delta,\Vert\mathcal{V}\Vert_{M^{q}(p;\omega_{k+1})}\right)$.
   By virtue of the structural properties of~$\mathcal{A}$ and the frequently used inequality:
   $$\big\Vert\nabla(v_{i}\varphi)\big\Vert^{p}_{L^{p}(\omega_{k};\mathbb{R}^{n})}\leq 2^{p-1}\Big(\big\Vert v_{i}\nabla \varphi\big\Vert^{p}_{L^{p}(\omega_{k};\mathbb{R}^{n})}+\big\Vert \varphi\nabla v_{i}\big\Vert^{p}_{L^{p}(\omega_{k};\mathbb{R}^{n})}\Big),$$
  we observe that for all~$i>k$ and all~$\varphi\in C^{\infty}_{c}(\omega_{k})$:
  \begin{equation*}
   	 \left((1\!-\!\varepsilon)\alpha_{\omega_{k}} \! \!- \! 2^{p-1}\delta\right)\!\left\Vert|\nabla v_{i}|\varphi\right\Vert_{L^{p}(\omega_{k})}^{p}
    \!\leq \!\! \left(\!\!\left(\!\frac{p-1}{\varepsilon} \! \right)^{p-1} \!\!\beta_{\omega_{k}} \! + \! 2^{p-1}\delta \!\right)\!\!\left\Vert v_{i}|\nabla \varphi|\right\Vert_{L^{p}(\omega_{k})}^{p}
   \!\! + \!C\left\Vert v_{i}\varphi\right\Vert^{p}_{L^{p}(\omega_{k})} \!,
  \end{equation*}
where~$C=C\left(n,p,q,\delta,\Vert\mathcal{V}\Vert_{M^{q}(p;\omega_{k+1})}\right)$.
  Let $\delta>0$ be such that~$(1-\varepsilon)\alpha_{\omega_{k}}-2^{p-1}\delta>0$, and fix $\omega\Subset\omega'\Subset\omega_{k}$. Choose~$\varphi\in C^{\infty}_{c}(\omega_{k})$ \cite[Theorem 1.4.1]{cutoff} such that$$\supp(\varphi)\subseteq\omega',\quad0\leq \varphi\leq 1~\mbox{in}~\omega',\quad \varphi=1~\mbox{in}~\omega, \mbox{ and } |\nabla \varphi|\leq C(\omega',\omega)~\mbox{in}~\omega'.$$ Consequently, with~$C'=C\left(n,p,q,\delta,\varepsilon,\alpha_{\omega_{k}},\Vert\mathcal{V}\Vert_{M^{q}(p;\omega_{k+1})}\right)$ and~$C''=C(p,\delta,\varepsilon,\alpha_{\omega_{k}},\beta_{\omega_{k}})$, we have
   \begin{eqnarray*}
    \Vert\nabla v_{i}\Vert_{L^{p}(\omega;\mathbb{R}^{n})}^{p}+\Vert v_{i}\Vert_{L^{p}(\omega)}^{p}
    &\leq& \Vert|\nabla v_{i}|\varphi\Vert_{L^{p}(\omega_{k})}^{p}+\Vert v_{i}\varphi\Vert_{L^{p}(\omega_{k})}^{p}\\
    &\leq& C'\big\Vert v_{i}\varphi\big\Vert^{p}_{L^{p}(\omega_{k})}
    + C''\big\Vert v_{i}|\nabla \varphi|\big\Vert_{L^{p}(\omega_{k})}^{p}
    \leq \tilde{C},
  \end{eqnarray*}
  where the positive constant $\tilde{C}$ does not depend on~$v_{i}$.
  So~$\{v_{i}\}_{i\in\mathbb{N}}$ is bounded in~$W^{1,p}(\omega)$. Hence, there exists a subsequence converging weakly in~$W^{1,p}(\omega)$ to the nonnegative function~$v\in W^{1,p}(\omega)$ with~$v(x_{0})=1$ because~$\{v_{i}\}_{i\in\mathbb{N}}$ converges uniformly to~$v$ in~$\omega$.

   The task is now to show that~$v$ is a positive solution of~$Q'_{p,\mathcal{A},\mathcal{V}}[u]=0$ in~$\tilde{\omega}\Subset\omega$ such that~$x_{0}\in\tilde{\omega}$. For any~$\psi\in C^{\infty}_{c}(\tilde{\omega})$, we have
   \begin{eqnarray*}
   &&\left\vert\int_{\tilde{\omega}}\mathcal{V}_{i}v_{i}^{p-1}\psi\dx-\int_{\tilde{\omega}}\mathcal{V}v^{p-1}\psi\dx\right\vert\\
   &=&\left\vert\int_{\tilde{\omega}}\mathcal{V}_{i}v_{i}^{p-1}\psi\dx-\int_{\tilde{\omega}}\mathcal{V}_{i}v^{p-1}\psi\dx+\int_{\tilde{\omega}}\mathcal{V}_{i}v^{p-1}\psi\dx-\int_{\tilde{\omega}}\mathcal{V}v^{p-1}\psi\dx\right\vert\\
   &\leq& \int_{\tilde{\omega}}\vert\mathcal{V}_{i}\vert \vert v_{i}^{p-1}-v^{p-1}\vert|\psi|\dx+\left\vert\int_{\tilde{\omega}}(\mathcal{V}_{i}-\mathcal{V}) v^{p-1}\psi\dx\right\vert\\
   &\leq& C(\psi)\int_{\tilde{\omega}}\vert\mathcal{V}_{i}\vert \vert v_{i}^{p-1}-v^{p-1}\vert\dx+\left\vert\int_{\tilde{\omega}}(\mathcal{V}_{i}-\mathcal{V}) v^{p-1}\psi\dx\right\vert.
   \end{eqnarray*}
The sequence~$\{v_{i}\}_{i\in\mathbb{N}}$ is uniformly bounded by Harnack's inequality in~$\tilde{\omega}$. The limit function~$v$ is continuous in~$\omega$ and hence bounded in~$\tilde{\omega}$. The function~$f(t)\triangleq t^{p-1}$ is uniformly continuous on~$[0,L]$ for any~$L>0$. Then~$\{v_{i}^{p-1}\}_{i\in\mathbb{N}}$ converges uniformly in~$\tilde{\omega}$ to~$v^{p-1}$ as~$\{v_{i}\}_{i\in\mathbb{N}}$ converges uniformly to~$v$. Furthermore, by a standard finite covering argument, because~$\mathcal{V}_{i}$ converges weakly to~$\mathcal{V}$ in~$M^{q}_{\loc}(p;\Omega)$, we infer that~$\int_{\tilde{\omega}}\vert\mathcal{V}_{i}\vert\dx$ is bounded with respect to~$i$. Hence,  $$\int_{\tilde{\omega}}\vert\mathcal{V}_{i}\vert \vert v_{i}^{p-1}-v^{p-1}\vert\dx\to 0 \qquad \mbox{ as } i\to \infty.$$
Moreover, by the weak convergence of $\{\mathcal{V}_{i}\}_{i\in\mathbb{N}}$ to
	$\mathcal{V}$, it follows that $$ \int_{\tilde{\omega}}(\mathcal{V}_{i}-\mathcal{V}) v^{p-1}\psi\dx\to 0 \qquad \mbox{ as } i\to\infty.$$
	Consequently, it follows that
\begin{equation}\label{potentialconvergence}
\lim_{i\rightarrow\infty}\int_{\tilde{\omega}}\mathcal{V}_{i}v_{i}^{p-1}\psi\dx=\int_{\tilde{\omega}}\mathcal{V}v^{p-1}\psi\dx.
\end{equation}

  {\bf Claim:} The sequence $\{\mathcal{A}(x,\nabla v_{i})\}$ converges weakly in~$L^{p'}(\tilde{\omega};\mathbb{R}^{n})$ to
  $ \mathcal{A}(x,\nabla v)$. This will  imply that for any~$\psi\in C^{\infty}_{c}(\tilde{\omega})$  we have
$$\int_{\tilde{\omega}}\mathcal{A}(x,\nabla v)\cdot\nabla \psi\dx+\int_{\tilde{\omega}}\mathcal{V}v^{p-1}\psi\dx=\lim_{i\to\infty}\int_{\tilde{\omega}}\left(\mathcal{A}(x,\nabla v_{i})\cdot\nabla \psi+\mathcal{V}_{i}v_{i}^{p-1}\psi\right)\dx=0.$$ In other words, $v$ is a nonnegative solution of the equation~$Q'_{p,\mathcal{A},\mathcal{V}}[u]=0$ in~$\tilde{\omega}$.

  To this end, choose~$\psi\in C^{\infty}_{c}(\omega)$ \cite[Theorem 1.4.1]{cutoff} such that$$\supp(\psi)\subseteq\omega,\quad0\leq \psi\leq 1~\mbox{in}~\omega,\quad \psi=1~\mbox{in}~\tilde{\omega}, \mbox{ and }  |\nabla \psi|\leq C(\tilde{\omega},\omega)~\mbox{in}~\omega.$$ Testing~$\psi(v_{i}-v)$ in the definition of~$v_{i}$ being a solution of~$Q'_{p,\mathcal{A},\mathcal{V}_{i}}[w]=0$ in~$\omega_{i}$, we get$$\int_{\omega}\psi\mathcal{A}(x,\nabla v_{i})\cdot \nabla(v_{i}-v)\dx=-\int_{\omega}(v_{i}-v)\mathcal{A}(x,\nabla v_{i})\cdot\nabla \psi\dx-\int_{\omega}\mathcal{V}_{i}v_{i}^{p-1}\psi(v_{i}-v)\dx.$$

   We claim that
   \begin{equation}\label{ineq7}
   \int_{\omega}\psi\mathcal{A}(x,\nabla v_{i})\cdot \nabla(v_{i}-v)\dx\rightarrow 0\mbox{ as } i\rightarrow \infty.
   \end{equation}
  As in the proof of \eqref{potentialconvergence}, $\int_{\omega}\mathcal{V}_{i}v_{i}^{p-1}\psi(v_{i}-v)\dx \to 0$ as $i\rightarrow\infty$.
In addition,
  \begin{eqnarray*}
    \left\vert-\int_{\omega}(v_{i}-v)\mathcal{A}(x,\nabla v_{i})\cdot\nabla \psi\dx\right\vert
    &\leq& \beta_{\omega}\int_{\omega}|\nabla v_{i}|^{p-1}|(v_{i}-v)\nabla \psi|\dx\\
    &\leq& \beta_{\omega}\Big(\int_{\omega}|(v_{i}-v)\nabla \psi|^{p}\dx\Big)^{1/p}\Vert\nabla v_{i}\Vert^{p/p'}_{L^{p}(\omega;\mathbb{R}^{n})}\\
    &\leq &C\left(\beta_{\omega},\omega,\tilde{\omega}, \psi\right)\Vert v_{i}-v\Vert_{L^{p}(\omega)}\Vert\nabla v_{i}\Vert^{p/p'}_{L^{p}(\omega;\mathbb{R}^{n})}.
  \end{eqnarray*}
  Because the norms $\Vert\nabla v_{i}\Vert^{{p}/{p'}}_{L^{p}(\omega;\mathbb{R}^{n})}$ are uniformly bounded for all~$i\in\mathbb{N}$, and~$v_{i}$ converges to~$v$ uniformly in~$\omega$ as~$i \to \infty$,
 we get $$\left\vert-\int_{\omega}(v_{i}-v)\mathcal{A}(x,\nabla v_{i})\cdot\nabla \psi\dx\right\vert \to 0 \mbox{ as } i\to \infty.$$
It follows that
\begin{eqnarray*}
0\leq \mathcal{I}_{i}&\triangleq&\int_{\tilde{\omega}}\big(\mathcal{A}(x,\nabla v_{i})-\mathcal{A}(x,\nabla v)\big)\cdot(\nabla v_{i}-\nabla v)\dx\\
&\leq&\int_{\omega}\psi\big(\mathcal{A}(x,\nabla v_{i})-\mathcal{A}(x,\nabla v)\big)\cdot(\nabla v_{i}-\nabla v)\dx \to 0 \mbox{ as } i\to \infty,
\end{eqnarray*}
which is derived from \eqref{ineq7} and the weak convergence of $\{\nabla v_{i}\}_{i\in\mathbb{N}}$ to~$\nabla v$.
Hence, $\lim_{i\rightarrow\infty}\mathcal{I}_{i}=0$. By means of~\cite[Lemma 3.73]{HKM}, we obtain in~$L^{p'}(\tilde{\omega};\mathbb{R}^{n})$,
$$\mathcal{A}(x,\nabla v_{i})\rightharpoonup \mathcal{A}(x,\nabla v)
\qquad \mbox{ as } i\to \infty.$$
Hence, $v$ is a positive solution of the equation~$Q'_{p,\mathcal{A},\mathcal{V}}[u]=0$ in~$\tilde{\omega}$ satisfying $v(x_{0})=1$.

In conclusion, for any~$\tilde{\omega}\Subset\omega\Subset\Omega$ with~$x_{0}\in \tilde{\omega}$, there exists a subsequence of~$\{v_{i}\}_{i\in\mathbb{N}}$ converging weakly and locally uniformly in~$\tilde{\omega}$ to a positive weak solution of the equation~$Q'_{p,\mathcal{A},\mathcal{V}}[u]=0$ in~$\tilde{\omega}$.  Using a standard diagonal argument, we may extract  a subsequence of~$\{v_{i}\}_{i\in\mathbb{N}}$ which converges weakly in~$W^{1,p}(\omega_{i})$ for all~$i\in\mathbb{N}$ and locally uniformly in~$\Omega$ to a positive weak solution~$v\in W^{1,p}_{\loc}(\Omega)$ of the equation~$Q'_{p,\mathcal{A},\mathcal{V}}[u]=0$ in~$\Omega$.
\end{proof}
\section{Generalized principal eigenvalue}\label{sec_eigenvalue}
Throughout the present section we consider solutions in a fixed domain $\gw\Subset \Gw$. First,  by virtue of the weak lower semicontinuity as well as the coercivity of certain functionals related to the functional $Q_{p,\mathcal{A},V}$, we prove that the generalized principal eigenvalue of the operator $Q'_{p,\mathcal{A},V}$ in $\gw$ is, in fact, a principal eigenvalue of $Q'_{p,\mathcal{A},V}$.  Moreover, the principal eigenvalue is simple, which is proved by virtue of the Picone-type identity (Lemma \ref{Picone}). After that, we show in Theorem \ref{maximum} (together with Theorem~\ref{complement}) that the following properties are equivalent: the positivity of the principal eigenvalue, the  validity of the generalized weak or strong maximum principles, and  the unique solvability of the Dirichlet problem $Q'_{p,\mathcal{A},V}[u]=g$ in~$W^{1,p}_{0}(\omega)$. We also establish a weak comparison principle, which is of core importance in Section~\ref{minimal}. See \cite{Pinchovergp} for more on the generalized principal eigenvalue.
\subsection{D\'{\i}az-Sa\'a type inequalities}
In this subsection, we generalize the D\'{\i}az-Sa\'a type inequalities as a counterpart of \cite[Lemma 3.3]{Pinchover},
see also \cite{Anane1987, Diaz, Lindqvist} for related results. {To this end, Assumption \ref{ass2}, concerning a stronger convexity requirement on~$|\xi|_{\mathcal{A}}^{p}$, is assumed.} The D\'{\i}az-Sa\'a type inequalities are used to prove
the uniqueness of solutions of two Dirichlet problems (see theorems~\ref{maximum} and \ref{5proposition}). But in Lemma \ref{newDiaz} and hence in Theorem \ref{5proposition}, this assumption is not supposed.
\begin{Def}
	{\em Assume that $\mathcal{A}$  satisfies Assumption \ref{ass8}, and let $V\in M_\loc^{q}(p;\Omega)$. Let $\omega\Subset\Omega$ be a subdomain.
	   A real number~$\lambda$ is called an \emph{eigenvalue with an eigenfunction v} of the Dirichlet eigenvalue problem
			\begin{equation}\label{evp}
			\begin{cases}
			Q'_{p,\mathcal{A},V}[u]=\lambda|u|^{p-2}u&\text{in~$\omega$},\\
			u=0& \text{on~$\partial\omega$},
			\end{cases}
			\end{equation}
			if~$v\in W^{1,p}_{0}(\omega)\setminus\{0\}$ is a solution of the equation $Q'_{p,\mathcal{A},V}[u]=\lambda |u|^{p-2}u$.
	}
\end{Def}
\begin{lem}[D\'{\i}az-Sa\'a type inequalities]\label{elementary}
	Suppose that~$\mathcal{A}$ satisfies Assumption~\ref{ass8} and  Assumption~\ref{ass2}. Let~$\omega\Subset\Omega$ be a domain and~$g_{i},V_{i}\in M^{q}(p;\omega)$, where~$i=1,2$. There exists a positive constant~${C_\gw(p,\mathcal A)}$ satisfying the following conclusions.
	\begin{enumerate}
	\item[$(1)$] Let $w_{i}\in W^{1,p}_{0}(\omega)\setminus\{0\}$ be nonnegative solutions of $Q'_{p,\mathcal{A},V_{i}}[u]=g_{i}$ in $\gw$, where $i=1,2$, and let~$w_{i,h}\triangleq w_{i}+h$, where~$h$ is a positive constant. Then
	\begin{multline*}
	I_{h,g_{1},g_{2},w_{1},w_{2}}
	\!\!\triangleq \!\!\int_{\omega}\!\!\left(\!\frac{g_{1} \! - \! V_{1}w_{1}^{p-1}}{w_{1,h}^{p-1}}-
	\frac{g_{2}\!-\!V_{2}w_{2}^{p-1}}{w_{2,h}^{p-1}}\!\right)\!\!(w^{p}_{1,h}-w^{p}_{2,h})\!\dx
	\!\geq \! {C_\gw(p,\mathcal A)}L_{h,w_{1},w_{2}},
	\end{multline*}
	where
	{$$L_{h,w_{1},w_{2}}\triangleq\int_{\omega}\left(w^{p}_{1,h}\bigg[\frac{\nabla w_{1,h}}{w_{1,h}},\frac{\nabla w_{2,h}}{w_{2,h}}\bigg]_{\mathcal{A}}+w^{p}_{2,h}\bigg[\frac{\nabla w_{2,h}}{w_{2,h}},\frac{\nabla w_{1,h}}{w_{1,h}}\bigg]_{\mathcal{A}}\right)\!\dx,$$
	and $[\xi,\eta]_{\mathcal{A}}$ is given by \eqref{eq_strict}.}
	\item[$(2)$]  Let $w_\gl$ and $w_\gm$ be nonnegative eigenfunctions  of the operators $Q'_{p,\mathcal{A},V_{i}}$ in $\gw$ with eigenvalues $\gl$ and $\gm$, respectively.  Then
	$$I_{0,g_{\lambda},g_{\mu},w_{\lambda},w_{\mu}}=\int_{\omega}\big((\lambda-\mu)-(V_{1}-V_{2})\big)(w_{\lambda}^{p}-w_{\mu}^{p})\dx\geq
	{C_\gw(p,\mathcal A)}L_{0,w_{\lambda},w_{\mu}}.$$
	\end{enumerate}
\end{lem}
\begin{proof}	
$(1)$ Let~$\psi_{1,h}\triangleq(w^{p}_{1,h}-w^{p}_{2,h})w_{1,h}^{1-p}$. Then~$\psi_{1,h}\in W^{1,p}_{0}(\omega)$. It follows that
\begin{equation*}
\int_{\omega}\mathcal{A}(x,\nabla w_{1})\cdot \nabla \psi_{1,h}\dx+\int_{\omega}V_{1}|w_{1}|^{p-2}w_{1}\psi_{1,h}\dx=\int_{\omega}g_{1}\psi_{1,h}\dx.
\end{equation*}
We thus get
\begin{eqnarray*}\label{w1h}
	&&\int_{\omega}(w^{p}_{1,h}-w^{p}_{2,h})\left\vert\frac{\nabla w_{1,h}}{w_{1,h}}\right\vert_{\mathcal{A}}^{p}\dx
	-p\int_{\omega}w_{2,h}^{p}\mathcal{A}\left(x,\frac{\nabla w_{1,h}}{w_{1,h}}\right)\cdot\left(\frac{\nabla w_{2,h}}{w_{2,h}}-\frac{\nabla w_{1,h}}{w_{1,h}}\right) \!\! \dx\\
	&=&\int_{\omega}\frac{g_{1}-V_{1}w_{1}^{p-1}}{w_{1,h}^{p-1}}(w^{p}_{1,h}-w^{p}_{2,h})
	\dx.
\end{eqnarray*}
Similarly, we see that
\begin{eqnarray*}\label{w2h}
	&&\int_{\omega}(w^{p}_{2,h}-w^{p}_{1,h})\left\vert\frac{\nabla w_{2,h}}{w_{2,h}}\right\vert_{\mathcal{A}}^{p}\dx
	-p\int_{\omega}w_{1,h}^{p}\mathcal{A}\left(x,\frac{\nabla w_{2,h}}{w_{2,h}}\right)\cdot\left(\frac{\nabla w_{1,h}}{w_{1,h}}-\frac{\nabla w_{2,h}}{w_{2,h}}\right)\!\!\dx\\
	&=&\int_{\omega}\frac{g_{2}-V_{2}w_{2}^{p-1}}{w_{2,h}^{p-1}}(w^{p}_{2,h}-w^{p}_{1,h})\dx.
\end{eqnarray*}
Adding the previous derived equalities yields
\begin{eqnarray*}
	& &I_{h,g_{1},g_{2},w_{1},w_{2}}
	\!=\!\int_{\omega}\!\!w^{p}_{1,h} \! \left(\left\vert\frac{\nabla w_{1,h}}{w_{1,h}}\right\vert_{\mathcal{A}}^{p}\!\!-\!\left\vert\frac{\nabla w_{2,h}}{w_{2,h}}\right\vert_{\mathcal{A}}^{p}
	\!-\! p\mathcal{A}\!\!\left(x,\frac{\nabla w_{2,h}}{w_{2,h}}\right)\!\!\cdot\!\!\left(\frac{\nabla w_{1,h}}{w_{1,h}}-\frac{\nabla w_{2,h}}{w_{2,h}}\right) \! \right)\!\!\dx\\[2mm]
	&+& \!\!\int_{\omega}w^{p}_{2,h}\left(\left\vert\frac{\nabla w_{2,h}}{w_{2,h}}\right\vert_{\mathcal{A}}^{p}-\left\vert\frac{\nabla w_{1,h}}{w_{1,h}}\right\vert_{\mathcal{A}}^{p}-p\mathcal{A}\left(x,\frac{\nabla w_{1,h}}{w_{1,h}}\right)\cdot\left(\frac{\nabla w_{2,h}}{w_{2,h}}-\frac{\nabla w_{1,h}}{w_{1,h}}\right)\right)\!\!\dx.
\end{eqnarray*}
{Applying Assumption~\ref{ass2},} we obtain the conclusion $(1)$.

$(2)$ Using part $(1)$, we have
\begin{eqnarray*}
	&&\left\vert\left((\lambda-V_{1})\left(\frac{w_{\lambda}}{w_{\lambda,h}}\right)^{p-1}-(\mu-V_{2})\left(\frac{w_{\mu}}{w_{\mu,h}}\right)^{p-1}\right)(w_{\lambda,h}^{p}-w_{\mu,h}^{p})\right\vert\\
	&\leq& \left(|\lambda-V_{1}|+|\mu-V_{2}|\right)\left((w_{\lambda}+1)^{p}+(w_{\mu}+1)^{p}\right)\\
	&\leq& 2^{p-1}\left(|\lambda-V_{1}|+|\mu-V_{2}|\right)(w_{\lambda}^{p}+w_{\mu}^{p}+2)\in L^{1}(\omega).
\end{eqnarray*}
On the other hand, for a.e.~$x\in\omega$, the limit
\begin{eqnarray*}
	&&\lim_{h\rightarrow 0}\left((\lambda-V_{1})\left(\frac{w_{\lambda}}{w_{\lambda,h}}\right)^{p-1}-(\mu-V_{2})\left(\frac{w_{\mu}}{w_{\mu,h}}\right)^{p-1}\right)(w_{\lambda,h}^{p}-w_{\mu,h}^{p})\\
	&=&(\lambda-\mu-V_{1}+V_{2})(w_{\lambda}^{p}-w_{\mu}^{p}).
\end{eqnarray*}
Hence, the dominated convergence theorem and Fatou's lemma imply $(2)$.
\end{proof}
\begin{lemma}\label{newDiaz}
  Assume that $\omega\Subset \Gw$ is a bounded Lipschitz domain, $\mathcal{A}$ satisfies Assumption~\ref{ass8}, $g_{i},V_{i}\in M^{q}(p;\omega)$.  For $i=1,2$, let $w_{i}\in W^{1,p}(\omega)$ be respectively,
		positive solutions of the equations $Q'_{p,\mathcal{A},V_{i}}[w]=g_{i}$ in $\gw$,  which are bounded away from zero in~$\omega$ and  satisfy $w_{1}=w_{2}>0$ on~$\partial\omega$ in the trace sense.
		Then
		$$I_{0,g_{1},g_{2},w_{1},w_{2}}=\int_{\omega}\left(\left(\frac{g_{1}}{w_{1}^{p-1}}-\frac{g_{2}}{w_{2}^{p-1}}\right)-(V_{1}-V_{2})\right)(w_{1}^{p}-w_{2}^{p})\dx
		\geq 0,$$ and~$I_{0,g_{1},g_{2},w_{1},w_{2}}=0$ if and only if $\nabla w_{1}/w_{1}=\nabla w_{2}/w_{2}$.
\end{lemma}
\begin{proof}
Letting~$h=0$, we see at once that the lemma follows from the proof of $(1)$ of Theorem \ref{elementary} and Lemma \ref{strictconvexity} (without assuming the stronger convexity).
\end{proof}
\subsection{Weak lower semicontinuity and coercivity}
In this subsection, we study the weak lower semicontinuity and coercivity of certain functionals related to the functional $Q_{p,\mathcal{A},V}$. See also \cite[Section 8.2]{Evans}.
\begin{Def}
	{\em
    Let~$(X,\Vert\cdot\Vert_{X})$ be a Banach space. A functional $J:X\to\mathbb{R}\cup\{\infty\}$ is said to be {\em coercive} if
    $J[u] \to \infty\mbox{ as }\Vert u\Vert_{X} \to   \infty.$

\noindent The functional~$J$ is said to be {\em (sequentially) weakly lower semicontinuous} if $$J[u] \leq  \liminf_{k\to\infty} J[u_{k}]  \qquad \mbox{ whenever }u_{k}\rightharpoonup u.$$}
  \end{Def}
\begin{notation}
		\emph{For a domain~$\omega\subseteq \Gw$,~$\mathcal{A}$ satisfying Assumption~\ref{ass8}, and $V\in M^{q}_\loc (p;\Omega)$, let $Q_{p,\mathcal{A},V}[\varphi;\omega]$ be the functional on $C^{\infty}_{c}(\omega)$ given by  $$\vgf\mapsto \int_{\omega}\Big(\vert\nabla \varphi\vert_{\mathcal{A}}^{p}+V\vert \varphi\vert^{p}\Big)\dx.$$ When~$\omega=\Omega$, we write $Q_{p,\mathcal{A},V}[\varphi]\triangleq Q_{p,\mathcal{A},V}[\varphi;\Omega].$}
\end{notation}
The next four theorems can be proved by standard arguments which are similar to the proof of \cite[propositions 3.6 and 3.7]{Pinchover}, and therefore their proofs are omitted. We state them as four separate theorems for the sake of clarity.
\begin{Thm}\label{ThmJ}
  Consider the domains~$\omega\Subset\omega'\Subset \Gw$, and let $\mathcal{A}$ satisfy Assumption \ref{ass8}, and $g,\mathcal{V}\in M^{q}(p;\omega')$, where~$\omega$ is Lipschitz. Then the  functional
  $$\bar{J}:W^{1,p}(\omega)\rightarrow\mathbb{R}\cup\{\infty\},\quad \bar{J}[u]\triangleq Q_{p,\mathcal{A},\mathcal{V}}[u;\omega]-\int_{\omega}g\vert u\vert\dx,$$ is weakly lower semicontinuous in~$W^{1,p}(\omega)$.
\end{Thm}
\begin{Thm}\label{ThmJ1}
  Consider the domains~$\omega\Subset \Gw$, and let $\mathcal{A}$ satisfy Assumption \ref{ass8}, $\mathcal{V}\in M^{q}(p;\omega)$, and~$g\in L^{p'}(\omega)$. Then the functional
  $$J:W^{1,p}_{0}(\omega)\rightarrow\mathbb{R}\cup\{\infty\},\quad J[u]\triangleq Q_{p,\mathcal{A},\mathcal{V}}[u;\omega]-\int_{\omega}gu\dx,$$ is weakly lower semicontinuous in~$W^{1,p}_{0}(\omega)$.
\end{Thm}
\begin{Thm}
   Consider the domains~$\omega\Subset\omega'\Subset \Gw$, where~$\omega$ is a Lipschitz domain. Let $\mathcal{A}$ satisfy Assumption \ref{ass8},  $g, \mathcal{V}\in M^{q}(p;\omega')$, and let~$\mathcal{V}$ be nonnegative. For any~$f\in W^{1,p}(\omega)$, ~$\bar{J}$ is coercive in$$\mathbf{A}\triangleq\{u\in W^{1,p}(\omega):u-f\in W^{1,p}_{0}(\omega)\}.$$
\end{Thm}
\begin{Thm}\label{thm-coercive}
  Consider a domain~$\omega\Subset\Omega$, and let $\mathcal{A}$ satisfy Assumption \ref{ass8}, $\mathcal{V}\in M^{q}(p;\omega)$, and~$g\in L^{p'}(\omega)$. If for some~$\varepsilon>0$ and all~$u\in W^{1,p}_{0}(\omega)$ we have,
  $$Q_{p,\mathcal{A},\mathcal{V}}[u;\omega]\geq\varepsilon\Vert u\Vert^{p}_{L^{p}(\omega)},$$
  then~$J$ is coercive in~$W^{1,p}_{0}(\omega)$.
\end{Thm}
\subsection{Picone identity}
This subsection concerns a Picone-type identity for $Q_{p,\mathcal{A},V}$. Picone's identities for the $(p,A)$-Laplacian and the $p$-Laplacian are crucial tools in \cite{Regev, Tintarev} (see also \cite{Regev1, Regev2}). In the present work, it will be used to give an alternative and direct proof (without Assumption~\ref{ass2}) of the AAP type theorem (see Lemma~\ref{lem_alter}), and in Theorem~\ref{newthm}.
\begin{lem}[{cf. \cite[ Lemma 2.2]{newpicone}}]\label{RL}
	Let~$\mathcal{A}$ satisfy Assumption \ref{ass8}, and define
	$$L(u,v)\triangleq |\nabla u|_{\mathcal{A}}^{p}+(p-1)\frac{u^{p}}{v^{p}}|\nabla v|^{p}_{\mathcal{A}}-p\frac{u^{p-1}}{v^{p-1}}\mathcal{A}(x,\nabla v)\cdot\nabla u,$$
	and
	$$R(u,v) \triangleq |\nabla u|_{\mathcal{A}}^{p}- \mathcal{A}(x,\nabla v)
	\cdot\nabla\left(\frac{u^{p}}{v^{p-1}}\right),$$
	where the functions~$u\in W^{1,p}_{\loc}(\Omega)$ and~$v\in W^{1,p}_{\loc}(\Omega)$ are respectively nonnegative and positive with
	$u^{p}/v^{p-1}\in W^{1,p}_{\loc}(\Omega)$ such that the product {and chain rules} for $u^{p}/v^{p-1}$ hold.
 Then $$L(u,v)(x)=R(u,v)(x) \qquad \mbox{for a.e.~$x\in \Omega$}.$$
Furthermore,
we have~$L(u,v) \geq 0$ a.e. in $\Omega$ and $L(u,v)=0$ a.e. in $\Omega$ if and only if $u=kv$ for some constant~$k\geq 0$.
\end{lem}
\begin{remark}
  \emph{The lemma concerns pointwise equality and inequality. Therefore, the proof in \cite[ Lemma 2.2]{newpicone} applies to our more general case where $\mathcal{A}$ depends also on $x$. Hence, the proof is omitted.}
\end{remark}
\begin{lem}[Picone-type identity]\label{Picone}
 Let $\mathcal{A}$ satisfy Assumption \ref{ass8} and  $V\!\in\!M^{q}_{\loc}(p;\Omega)$. For any positive
	solution~$v\in W^{1,p}_{\loc}(\Omega)$ of $Q'_{p,\mathcal{A},V}[w]=0$ in~$\Omega$, and any nonnegative function~$u\in W^{1,p}_{c}(\Omega)$ with $u^{p}/v^{p-1}\in W^{1,p}_{c}(\Omega)$ such that the product {and chain rules} for $u^{p}/v^{p-1}$ hold, we have
	$$Q_{p,\mathcal{A},V}[u]=\int_{\Omega} L(u,v)(x)\dx.$$ If instead,~$v\in W^{1,p}_{\loc}(\Omega)$ is either a positive subsolution or a positive supersolution, and all the other conditions are satisfied, then respectively,
	$$Q_{p,\mathcal{A},V}[u]\leq\int_{\Omega} L(u,v)(x)\dx, \quad
	\mbox{or} \quad   Q_{p,\mathcal{A},V}[u]\geq\int_{\Omega} L(u,v)(x)\dx.$$
\end{lem}
\begin{proof}
	The proof is similar to that of \cite[Proposition 3.3]{Regev}, and therefore it is omitted.
	\end{proof}
\begin{lem}\label{lem_alter}
	Let~$\mathcal{A}$ satisfy Assumption \ref{ass8}, and let~$V\in M^{q}_{\loc}(p;\Omega)$.
	\begin{enumerate}
		\item[$(1)$] For any positive
		solution~$v\in W^{1,p}_{\loc}(\Omega)$ of $Q'_{p,\mathcal{A},V}[\psi]=0$ in~$\Omega$ and any nonnegative function~$u\in W^{1,p}_{c}(\Omega)$ such that~$u^{p}/v^{p-1}\in W^{1,p}_{c}(\Omega)$ and the product {and chain rules} for~$u^{p}/v^{p-1}$ hold, if~$vw$ satisfies the product rule for~$w\triangleq u/v$, then
		$$Q_{p,\mathcal{A},V}[vw]=\int_{\Omega} \left(|v\nabla w+w\nabla v|^{p}_{\mathcal{A}}-w^{p}|\nabla v|^{p}_{\mathcal{A}}-pw^{p-1}v\mathcal{A}(x,\nabla v)\cdot\nabla w\right)\dx.$$
		\item[$(2)$] For a positive subsolution or a positive supersolution~$v\in W^{1,p}_{\loc}(\Omega)$ of $Q'_{p,\mathcal{A},V}[\psi]=0$ in~$\Omega$ and any nonnegative function~$u\in W^{1,p}_{c}(\Omega)$ such that~$u^{p}/v^{p-1}\in W^{1,p}_{c}(\Omega)$ and the
  product {and chain rules} for~$u^{p}/v^{p-1}$ hold, if~$vw$ satisfies the product rule for~$w\triangleq u/v$, then, respectively,
		$$Q_{p,\mathcal{A},V}[vw]\leq\int_{\Omega}\left(|v\nabla w+w\nabla v|^{p}_{\mathcal{A}}-w^{p}|\nabla v|^{p}_{\mathcal{A}}-pw^{p-1}v\mathcal{A}(x,\nabla v)\cdot\nabla w \right) \dx,$$ or
		$$Q_{p,\mathcal{A},V}[vw]\geq\int_{\Omega}\left(|v\nabla w+w\nabla v|^{p}_{\mathcal{A}}-w^{p}|\nabla v|^{p}_{\mathcal{A}}-pw^{p-1}v\mathcal{A}(x,\nabla v)\cdot\nabla w \right) \dx.$$
		\item[$(3)$]If $v\in W^{1,p}_{\loc}(\Omega)$ is either a positive
		solution or a positive supersolution of $Q'_{p,\mathcal{A},V}[u]=0$ in~$\Omega$, then the functional
		$Q_{p,\mathcal{A},V}$ is nonnegative on $W^{1,p}_{c}(\Omega)$.
	\end{enumerate}
\end{lem}
\begin{Rem}
	\emph{The third part of the lemma gives an alternative proof of~$(2)\Rightarrow (1)$ and~$(3)\Rightarrow (1)$ of the AAP type theorem (Theorem \ref{thm_AAP}).}
\end{Rem}
\begin{proof}[Proof of Lemma~\ref{lem_alter}]
	For the first two parts of the lemma, we apply the product rule directly in the final equality/inequalities of Lemma \ref{Picone}. The third part follows from the first two parts, the strict convexity of the function $|\cdot|^{p}_{\mathcal{A}}$, and an approximation argument. For details, see \cite[Theorem 5.2]{Regev} and \cite[Theorem 2.3]{Tintarev}.
\end{proof}
\subsection{Principal eigenvalues in domains~$\omega\Subset\Omega$}\label{eigenvalueunique}
\begin{Def}{\em
Let~$\mathcal{A}$ satisfy Assumption \ref{ass8} and let~$V\in M^{q}_{\loc}(p;\Omega)$. The \emph{generalized principal eigenvalue} of $Q'_{p,\mathcal{A},V}$ in a domain $\gw \subseteq \Gw$
is defind by
$$\lambda_{1}=\lambda_{1}(Q_{p,\mathcal{A},V};\omega)\triangleq\inf_{u\in C^{\infty}_{c}(\omega) \setminus\{0\}}\frac{Q_{p,\mathcal{A},V}[u;\omega]}{\Vert u\Vert_{L^{p}(\omega)}^{p}}\,.$$}	
\end{Def}
\begin{Rem}\label{rem_lambda1}{\em
	It follows that for a domain $\gw\Subset\Gw$  we have
\begin{equation}\label{eq_pev}
  \lambda_{1}(Q_{p,\mathcal{A},V};\omega) = \inf_{u\in W^{1,p}_{0}(\omega)\setminus\{0\}}\frac{Q_{p,\mathcal{A},V}[u;\omega]}{\Vert u\Vert_{L^{p}(\omega)}^{p}}\,.
\end{equation}
}
\end{Rem}
\begin{lemma}\label{easylemma}
  Let $\omega\Subset\Omega$ be a domain, let $\mathcal{A}$ satisfy Assumption \ref{ass8}, and let $V\in M^{q}(p;\omega)$. All eigenvalues of \eqref{evp} are larger than or equal to $\lambda_{1}$.
\end{lemma}
\begin{proof}
 Testing any eigenfunction, we get the conclusion.
\end{proof}
\begin{Def}
	{\em A \emph{principal eigenvalue} of \eqref{evp} is an eigenvalue with a nonnegative eigenfunction, which is called a \emph{principal eigenfunction}.}
\end{Def}
We first state a useful lemma.
\begin{lemma}\label{functionalcv}
 Let~$\mathcal{A}$ satisfy Assumption~\ref{ass8}, and let $V\in M^{q}_{\loc}(p;\Omega)$. For every domain~$\omega\Subset\Omega$, if~$u_{k}\rightarrow u$ as~$k\rightarrow\infty$ in~$W^{1,p}_{0}(\omega)$, then
 $\displaystyle{\lim_{k\rightarrow\infty}}Q_{p,\mathcal{A},V}[u_{k}]=Q_{p,\mathcal{A},V}[u]$.
\end{lemma}
\begin{proof}
  By \cite[Lemma 5.23]{HKM}, we get~$\lim_{k\rightarrow\infty}\int_{\omega}|\nabla u_{k}|_{\mathcal{A}}^{p}\dx=\int_{\omega}|\nabla u|_{\mathcal{A}}^{p}\dx$. The elementary inequality
  	$$|x^{p}-y^{p}|\leq p|x-y|(x^{p-1}+y^{p-1}) \qquad  \forall x,y\geq0,$$
  the H\"{o}lder inequality, and the Morrey-Adams theorem, imply
  \begin{eqnarray*}
    \left\vert\int_{\omega}\!V(|u_{k}|^{p}-|u|^{p})\!\dx \right\vert&\!\!\leq\! \!& \int_{\omega}\!|V|||u_{k}|^{p}-|u|^{p}|\!\dx
    \!\leq p\!\int_{\omega}\!|V||u_{k}-u|||u_{k}|^{p-1}+|u|^{p-1}|\!\dx\\
    &\!\!\leq\!\!& C(p)\!\left(\int_{\omega}\!\!|V||u_{k}\!-\! u|^{p}\!\dx\!\!\right)^{\!1/p}\!\!\left(\int_{\omega}\!|V||u_{k}|^{p} \!+\! |V||u|^{p}\!\dx\!\right)^{\!1/p'}\!
    \underset{ k\rightarrow\infty}{\rightarrow 0}.
 \end{eqnarray*}
 Hence, the desired convergence follows.
\end{proof}
Now we prove that in every domain~$\omega\Subset\Omega$, the generalized principal eigenvalue is a principal and simple eigenvalue, whose uniqueness is proved in Corollary \ref{newuniqueness}.
\begin{Thm}\label{principaleigenvalue}
  Let $\omega\Subset\Omega$ be a domain, let $\mathcal{A}$ satisfy Assumption~\ref{ass8},
   and let $V\in M^{q}(p;\omega)$.
  \begin{enumerate}
   \item[$(1)$] The generalized principal eigenvalue is a principal eigenvalue of the operator~$Q'_{p,\mathcal{A},\mathcal{V}}$.
   \item[$(2)$] The principal eigenvalue is simple, i.e., for any two eigenfunctions $u$ and~$v$ associated with the eigenvalue $\gl_1$, we have $u=cv$ for some~$c\in\R$.
   \end{enumerate}
 \end{Thm}
  \begin{proof}[Proof of Theorem \ref{principaleigenvalue}]
	$(1)$ Applying the Morrey-Adams theorem (Theorem~\ref{MA_thm}) with the positive number~$\delta=\alpha_{\omega}$ and the ellipticity condition \eqref{structure}, we obtain that
  $$\lambda_{1}\geq -C(n,p,q)\alpha_{\omega}^{-n/(pq-n)}\Vert V\Vert^{pq/(pq-n)}_{M^{q}(p;\omega)}>-\infty.$$
  For any~$\varepsilon>0$, letting~$\mathcal{V}=V-\lambda_{1}+\varepsilon$, we immediately see that for all~$u\in W^{1,p}_{0}(\omega)$, $$Q_{p,\mathcal{A},\mathcal{V}}[u;\omega]\geq \varepsilon\Vert u\Vert_{L^{p}(\omega)}^{p}.$$
  Therefore, the functional~$Q_{p,\mathcal{A},V-\lambda_{1}+\varepsilon}[\;\cdot\; ;\omega]$ is coercive and weakly lower semicontinuous in $W^{1,p}_{0}(\omega)$, and hence also in~$W^{1,p}_{0}(\omega)\cap\{\Vert u\Vert_{L^{p}(\omega)}=1\}$. Therefore,  the infimum in \eqref{eq_pev}
  is attained in~$ W^{1,p}_{0}(\omega)\setminus\{0\}$.

  Let~$v\in W^{1,p}_{0}(\omega)\setminus\{0\}$ be a minimizer of \eqref{eq_pev}.
  By standard variational calculus techniques, we conclude that the minimizer $v\in W^{1,p}_{0}(\omega)$ satisfies the equation $Q'_{p,\mathcal{A},V}[u]=\lambda_1|u|^{p-2}u$ in the weak sense. Note that~$|v|\in W^{1,p}_{0}(\omega)$. In addition, almost everywhere in~$\omega$, we have~$\big|\nabla(|v|)\big|=|\nabla v|$ and~$\big|\nabla(|v|)\big|_{\mathcal{A}}=|\nabla v|_{\mathcal{A}}.$ Thus~$|v|$ is also a minimizer, and therefore, it satisfies the equation $Q'_{p,\mathcal{A},V}[u]=\lambda_1|u|^{p-2}u$ in the weak sense. So~$\lambda_{1}$ is a principal eigenvalue. The Harnack inequality and H\"{o}lder estimates guarantee that~$|v|$ is strictly positive and continuous in~$\omega$. Therefore, we may assume that $v>0$.

  $(2)$ The proof is inspired by \cite[Theorem 2.1]{Regev1}. Let $v,u\in W^{1,p}_{0}(\omega)$ be, respectively, a positive principal eigenfunction and any eigenfunction associated with~$\lambda_{1}$. It suffices to show~$u=cv$ for some~$c\in\R$. By part $(1)$ we may assume that $u>0$ in~$\omega$. Let~$\{\varphi_{k}\}_{k\in\mathbb{N}}\subseteq C^{\infty}_{c}(\omega)$ a nonnegative sequence approximating $u$ in~$W^{1,p}_{0}(\omega)$ and a.e. in~$\omega$. Then the product and chain rules for~$\varphi_{k}^{p}/v^{p-1}$ hold for all~$k\in\mathbb{N}$. By Lemma \ref{RL}, we get, for all~$k\in\mathbb{N}$,
	$$\int_{\omega} L(\varphi_{k},v)(x)\dx=\int_{\omega}|\nabla \varphi_{k}|_{\mathcal{A}}^{p}\dx-\int_{\omega}\mathcal{A}(x,\nabla v)
	\cdot\nabla\left(\frac{\varphi_{k}^{p}}{v^{p-1}}\right)\dx.$$
	Since~$\varphi_{k}^{p}/v^{p-1}\in W^{1,p}_{c}(\omega)$, we obtain
	$$\int_{\omega}\mathcal{A}(x,\nabla v)\cdot\nabla\left(\frac{\varphi_{k}^{p}}{v^{p-1}}\right)\dx+\int_{\omega}(V-\lambda_{1})v^{p-1}\frac{\varphi_{k}^{p}}{v^{p-1}}\dx=0.$$
	It follows that~$Q_{p,\mathcal{A},V-\lambda_{1}}[\varphi_{k};\gw] =\int_{\omega} L(\varphi_{k},v)(x)\dx.$ 
 By Fatou's lemma and Lemma \ref{functionalcv}, we obtain
\begin{eqnarray*}
0 & \leq &\int_{\omega} L(u,v)(x)\dx
 \leq \int_{\omega}\liminf_{k\rightarrow\infty}L(\varphi_{k},v)(x)\dx \leq
	\liminf_{k\rightarrow\infty}\int_{\omega}L(\varphi_{k},v)(x)\dx\\
&=&\liminf_{k\rightarrow\infty}Q_{p,\mathcal{A},V-\lambda_{1}}[\varphi_{k};\gw]
=Q_{p,\mathcal{A},V-\lambda_{1}}[u;\gw]=0.
\end{eqnarray*}
Lemma \ref{RL} and the connectedness of  $\gw$ imply that $u=cv$ in $\gw$ for some~$c>0$.
\end{proof}
\subsection{Positivity of the principal eigenvalues}\label{localtheory}
In this subsection, we consider positivity features of the operator $Q'_{p,\mathcal{A},V}$ in a {\em Lipschitz} domain
$\gw\Subset \Gw$. In particular, we study the relationship between the validity of the generalized strong/weak maximum principles, the existence of a proper positive supersolution, the unique solvability in $W^{1,p}_{0}(\omega)$ of the nonnegative Dirichlet problem $Q'_{p,\mathcal{A},V}[u]=g \geq 0$ in $\gw$, and the positivity of the principal eigenvalue.
 \begin{Def}
 \emph{Let $\omega$ be a bounded Lipschitz domain. A function $v\in W^{1,p}(\omega)$ is said to be \emph{nonnegative} on~$\partial\omega$ if $v^{-}\in W^{1,p}_{0}(\omega)$. A function~$v$ is said to be \emph{zero} on~$\partial\omega$ if $v\in W^{1,p}_{0}(\omega)$.}
  \end{Def}
 \begin{Def}
\emph{ Let $\omega\Subset\Omega$ be a Lipschitz domain, $\mathcal{A}$ satisfy Assumption~\ref{ass8}, and let $V\in M^{q}(p;\omega)$.
\begin{itemize}
\item The operator~$Q'_{p,\mathcal{A},V}$ is said to satisfy the \emph{generalized weak maximum principle in $\gw$} if every solution~$v  \in W^{1,p}(\omega)$ of the equation  $Q'_{p,\mathcal{A},V}[u]=g$ in $\gw$ with $0\leq g\in L^{p'}(\omega)$ and $v\geq 0$ on~$\partial\omega$ is nonnegative in~$\omega$;
\item the operator~$Q'_{p,\mathcal{A},V}$ satisfies the \emph{generalized strong maximum principle in $\gw$} if any  such a solution $v$ is either the zero function or strictly positive in~$\omega$.
\end{itemize}}
 \end{Def}
 Under Assumption~\ref{ass2}, by Theorem \ref{complement}, all the assertions in the following theorem are in fact equivalent even though we can not prove it completely at this point.
\begin{Thm}\label{maximum}
  Let~$\omega\Subset\Omega$, where~$\omega$ is a Lipschitz domain,~$\mathcal{A}$ satisfy Assumption~\ref{ass8}, and $V\in M^{q}(p;\omega)$. Consider the following assertions:
  \begin{enumerate}
    \item[$(1)$] The operator~$Q'_{p,\mathcal{A},V}$ satisfies the generalized weak maximum principle in~$\omega$.
    \item[$(2)$] The operator~$Q'_{p,\mathcal{A},V}$ satisfies the generalized strong maximum principle in~$\omega$.
    \item[$(3)$] The principal eigenvalue $\lambda_{1} =\lambda_{1}(Q_{p,\mathcal{A},V};\omega)$ is positive.
    \item[$(4)$] The equation $Q'_{p,\mathcal{A},V}[u]=0$ has a proper positive supersolution in~$W^{1,p}_{0}(\omega)$.
    \item[$(4')$] The equation $Q'_{p,\mathcal{A},V}[u]=0$ has a proper positive supersolution in~$W^{1,p}(\omega)$.
    \item[$(5)$] For any nonnegative~$g\in L^{p'}(\omega)$, there exists a  nonnegative solution  $v\in W^{1,p}_{0}(\omega)$ of the equation~$Q'_{p,\mathcal{A},V}[u]=g$ in~$\omega$ which is either zero or positive.
  \end{enumerate}
Then $(1)\Leftrightarrow (2)\Leftrightarrow (3)\Rightarrow (4)\Rightarrow (4')$, and~$(3)\Rightarrow (5)\Rightarrow (4).$

\medskip

Furthermore,
\begin{enumerate}
\item[$(6)$] If Assumption~\ref{ass2} is satisfied and $\gl_1>0$, then the solution in $(5)$ is unique.
\end{enumerate}
\end{Thm}
\begin{proof}
$(1)\Rightarrow (2)$ The generalized weak maximum principle implies that any solution~$v$ of $Q'_{p,\mathcal{A},V}[u]=g$ with $g\geq 0$, which is nonnegative on~$\partial\omega$, is nonnegative in~$\omega$. Then $v$ is a nonnegative supersolution of \eqref{half}.  The weak Harnack inequality implies that either $v>0$ or $v=0$ in $\omega$.

$(2)\Rightarrow (3)$ Let~$\lambda_{1}\leq 0$ and $v>0$ be its principal eigenfunction. By the homogeneity, the function $w=-v$ satisfies  $Q'_{p,\mathcal{A},V}[w]=\gl_1|w|^{p-1}w$, and $w=0 $ on $\partial\omega$ in the weak sense,  but this contradicts the generalized strong maximum principle.

$(3)\Rightarrow (1)$ Let~$v$ satisfy $v$ of $Q'_{p,\mathcal{A},V}[u]=g$  with $g\geq 0$, and $v\geq 0$ on~$\partial\omega$. Suppose that $v^{-}\neq 0$. Testing~$v^{-}$ in the definition of the solution of $Q'_{p,\mathcal{A},\mathcal{V}}[u] =g\geq 0$ , we get $$Q_{p,\mathcal{A},\mathcal{V}}[v^{-};\omega]=\int_{\{x\,\in\,\omega\,:\,v\,<\,0\}}gv\dx\leq 0.$$
	 Therefore, $\lambda_{1}\leq 0$, which  contradicts the assumption.

$(3)\Rightarrow (4)$ Since  $\lambda_{1}>0$, its principal eigenfunction is a proper positive supersolution of \eqref{half} in~$\omega$.

$(4)\Rightarrow (4')$ This implication follows from~$W^{1,p}_{0}(\omega)\subseteq W^{1,p}(\omega)$.

$(3)\Rightarrow (5)$ By Theorems \ref{ThmJ1} and \ref{thm-coercive}, the functional
 $J[u]= Q_{p,\mathcal{A},V}[u;\omega]- p\int_{\omega}gu\dx $
   is weakly lower semicontinuous and coercive in~$W^{1,p}_{0}(\omega)$ for $g\in L^{p'}(\gw)$. Therefore, the functional~$J$ has a minimizer in~$W^{1,p}_{0}(\omega)$ (see for example \cite[Theorem 1.2]{Struwe}). Consequently, the corresponding equation~$Q'_{p,\mathcal{A},V}[u]=g$ has a weak solution $v_{1}\in W^{1,p}_{0}(\omega)$. Note that~$(3)\Rightarrow (2)$. Therefore,  the solution~$v_{1}$ is either zero or positive in~$\omega$.

$(5)\Rightarrow (4)$  Use $(5)$ with~$g = 1$ to obtain a proper positive supersolution.

 $(6)$ Assume now that Assumption~\ref{ass2} is satisfied, and let us prove that $v_1=v$ is unique.  If~$v_{1}=0$, then~$g=0$. Hence, $Q_{p,\mathcal{A},V}[v;\omega]=0$, but this contradicts the assumption that $\gl_1 >0$.

Assume now that~$v_1>0$. Let~$v_{2}\in W^{1,p}_{0}(\omega)$ be any other positive solution. By part $(1)$ of Lemma~\ref{elementary} with~$g_{i}=g,V_{i}=V$  and $i=1,2$, we conclude
\begin{equation*}
\int_{\omega}V\!\!\left(\left(\frac{v_{1}}{v_{1,h}}\right)^{p-1}\!\!\! -
\!\!\left(\frac{v_{2}}{v_{2,h}}\right)^{p-1}\!\right)\!\!\left(v_{1,h}^{p}-v_{2,h}^{p}\right)\!\!\dx
\leq \!  \int_{\omega} \! g\left(\!\frac{1}{v_{1,h}^{p-1}}-\frac{1}{v_{2,h}^{p-1}} \!\right)\!\!
\left(v_{1,h}^{p}-v_{2,h}^{p}\right)\!\!\dx
\leq 0.
\end{equation*}
We note that  $$\lim_{h\rightarrow 0}V\left(\left(\frac{v_{1}}{v_{1,h}}\right)^{p-1}-\left(\frac{v_{2}}{v_{2,h}}\right)^{p-1}\right)\left(v_{1,h}^{p}-v_{2,h}^{p}\right)=0,$$
and
\begin{eqnarray*}
\left\vert V\left(\left(\frac{v_{1}}{v_{1,h}}\right)^{\!p-1}\!\!-\left(\frac{v_{2}}{v_{2,h}}\right)^{\!p-1}\right)\left(v_{1,h}^{p}-v_{2,h}^{p}\right)\right\vert
\leq 2|V|\left(\left(v_{1}+1\right)^{p}+\left(v_{2}+1\right)^{p}\right)
\in L^{1}(\omega).
\end{eqnarray*}
It follows that$$\lim_{h\rightarrow 0}\int_{\omega}g\left(\frac{1}{v_{1,h}^{p-1}}-\frac{1}{v_{2,h}^{p-1}}\right)\left(v_{1,h}^{p}-v_{2,h}^{p}\right)\dx=0.$$
By Fatou's lemma, and Lemma~\ref{elementary},  we  infer that $L_{0,v_1,v_2}\!=\!0$. Hence, $v_{2}\!=\!v_{1}$ in~$\omega$.
\end{proof}
\subsection{Weak comparison principle}\label{WCP}
\subsubsection{Super/sub-solution technique}
The following two theorems can be obtained by similar arguments to those of \cite[Lemma 5.1 and Proposition 5.2]{Pinchover}. We first state a weak comparison principle under the assumption that the potential is nonnegative.
\begin{lem}\label{5lemma}
	Let $\omega\Subset\Omega$ be a Lipschitz domain, $\mathcal{A}$ satisfy Assumption \ref{ass8}, $g\in M^{q}(p;\omega)$, and~$\mathcal{V}\in M^{q}(p;\omega)$, where $\mathcal{V}$ is nonnegative. For any subsolution~$v_{1}$ and any supersolution~$v_{2}$ of the equation $Q'_{p,\mathcal{A},\mathcal{V}}[u]=g$, in $\omega$ with $v_{1},v_{2}\in W^{1,p}(\omega)$ satisfying $\left(v_{2}-v_{1}\right)^{-}\in W^{1,p}_{0}(\omega)$, we have
		$v_{1}\leq v_{2}\quad \mbox{in}~\omega$.
\end{lem}
\begin{proof} Testing the integral inequalities for the subsolution~$v_{1}$ and the supersolution~$v_{2}$ with ~$\left(v_{2}-v_{1}\right)^{-}$ and then subtracting one from the other, we arrive at
	\begin{equation*}
	\int_{\omega}\left(\mathcal{A}(x,\nabla v_{1})-\mathcal{A}(x,\nabla v_{2})\right)\cdot\nabla\left((v_{2}-v_{1})^-\right)\dx
	+\int_{\omega}\mathcal{V}v_{1,2}\left(v_{2}-v_{1}\right)^{-}\dx\leq 0,
	\end{equation*}
where~$v_{1,2}\triangleq|v_{1}|^{p-2}v_{1}-|v_{2}|^{p-2}v_{2}$. 
	It follows that
	\begin{equation*}
	\int_{\{v_{2}<v_{1}\}}\left(\mathcal{A}(x,\nabla v_{1})-\mathcal{A}(x,\nabla v_{2})\right)\cdot\left(\nabla v_{1}-\nabla v_{2}\right)\dx\\
	+\int_{\{v_{2}<v_{1}\}}\mathcal{V}v_{1,2}\left(v_{1}-v_{2}\right)\dx\leq 0.
	\end{equation*}
	Since the above two terms are nonnegative, the two integrals are zero. Hence,
	$\nabla (v_{2}  - v_{1})^-\!=\!0$ a.e. in $\omega$. Therefore, the negative part of $v_{2} -v_{1}$ is a constant a.e. in $\omega$. Hence,  Poincar\'{e}'s inequality implies
	$(v_{2}\! - \! v_{1})^{-} \! = \! 0$  a.e. in $\omega$. Namely, $v_{1}\leq v_{2}$  a.e. in $\omega$.
\end{proof}
In order to establish a weak comparison principle when $V$ is not necessarily nonnegative, we need the following result which is of independent interest.
\begin{Thm}[Super/sub-solution technique]\label{5proposition}
	Let~$\omega\Subset\Omega$ be a Lipschitz domain, let~$\mathcal{A}$ satisfy Assumption~\ref{ass8}, and let $g,V\in M^{q}(p;\omega)$, where~$g$ is nonnegative a.e.~in~$\omega$. We assume that~$f,\varphi,\psi\in W^{1,p}(\omega)\cap C(\bar{\omega}),$ where~$f\geq 0$ a.e. in~$\omega$, and
	\[\begin{cases}
	Q'_{p,\mathcal{A},V}[\psi]\leq g\leq Q'_{p,\mathcal{A},V}[\varphi]&\text{in~$\omega$ in the weak sense,}\\
	\psi\leq f\leq \varphi& \text{on~$\partial\omega$,}\\
	0\leq \psi\leq\varphi&\text{in~$\omega$.}
	\end{cases}\]
	Then there exists a nonnegative function $u\in W^{1,p}(\omega)\cap C(\bar{\omega})$ satisfying
	\[\begin{cases}
	Q'_{p,\mathcal{A},V}[u]=g&\text{in~$\omega$,}\\
	u=f& \text{on~$\partial\omega$,}
	\end{cases}\]
	 and $\psi\leq u\leq \varphi$ in~$\omega$.
	 Moreover, if $f>0$ a.e. on~$\partial\omega$, then the above boundary value problem has a unique positive solution.
\end{Thm}
\begin{proof}
	Set$$\mathcal{K}\triangleq\left\{v\in W^{1,p}(\omega)\cap C(\bar{\omega}):0\leq \psi\leq v\leq \varphi \mbox{ in } \omega\right\}.$$
	For every~$x\in\omega$ and~$v\in\mathcal{K}$, let $G(x,v)\triangleq g(x)+2V^{-}(x)v(x)^{p-1}$. Then~$G\in M^{q}(p;\omega)$ and~$G\geq 0$ a.e. in~$\omega$.
	Let  the functionals
	$J$,$\bar{J}:W^{1,p}(\omega)\rightarrow \mathbb{R}\cup\{\infty\}$, be as in Theorems \ref{ThmJ} and \ref{ThmJ1} with~$|V|$ and~$G(x,v)$ as the potential and the right hand side, respectively. Choose a sequence~$\{u_{k}\}_{k\in\mathbb{N}}$ in
	$$\mathbf{A}\triangleq\{u\in W^{1,p}(\omega):u=f \mbox{ on }\partial\omega\},$$
	satisfying
	$$J[u_{k}]\downarrow m\triangleq \inf_{u\in\mathbf{A}}J[u].$$
	Because~$f\geq 0$,~$\{|u_{k}|\}_{k\in\mathbb{N}}\subseteq \mathbf{A},$ we infer $$m\leq J[|u_{k}|]=\bar{J}[u_{k}]\leq J[u_{k}],$$
	where the last inequality is on account of~$G(x,v)\geq 0$ a.e. in~$\omega$. Then~$\lim_{k\rightarrow\infty}\bar{J}[u_{k}]=m$. It is also immediate that~$\inf_{u\in\mathbf{A}}\bar{J}[u]=m$. On the other hand,~$\bar{J}$ is weakly lower semicontinuous and coercive. Noting that~$\mathbf{A}$ is weakly closed, it follows from  \cite[Theorem 1.2]{Struwe} that~$m$ is attained by a nonnegative function~$u_0\in\mathbf{A}$, that is,~$\bar{J}[u_0]=m$. In addition,~$J[u_0]=\bar{J}[u_0]=m$. Then~$u_0$ is a solution of
	\[\begin{cases}\label{problem1}
	Q'_{p,\mathcal{A},|V|}[u]=G(x,v)&\text{in~$\omega$,}\\
	u=f& \text{in the trace sense on~$\partial\omega$,}
	\end{cases}\]
	and for any~$v\in\mathcal{K}$, let~$T(v)$ be a solution of this Dirichlet problem.
	Then the map
	$$T:\mathcal{K}\longrightarrow W^{1,p}(\omega),$$ is increasing.
	Indeed, pick any~$v_{1}\leq v_{2}$ in~$\mathcal{K}$. Because~$G(x,v_{1})\leq G(x,v_{2})$, we infer that~$T(v_{1})$ and $T(v_{2})$ are respectively a solution and a supersolution of$$Q'_{p,\mathcal{A},|V|}[u]=G(x,v_{1}).$$ On~$\partial\omega$, we have~$T(v_{1})=f=T(v_{2})$. By Lemma \ref{5lemma}, we obtain~$T(v_{1})\leq T(v_{2})$ in~$\omega$.
	Consider any subsolution~$v\in W^{1,p}(\omega)\cap C(\bar{\omega})$ of the boundary value problem
	\[\begin{cases}
	Q'_{p,\mathcal{A},V}[u]=g&\text{in~$\omega$,}\\
	u=f& \text{on~$\partial\omega$.}
	\end{cases}\]
	Then in the weak sense in~$\omega$,$$Q'_{p,\mathcal{A},|V|}[v]=Q'_{p,\mathcal{A},V}[v]+G(x,v)-g(x)\leq G(x,v).$$ Furthermore, $T(v)$ is a solution of
	\[\begin{cases}
	Q'_{p,\mathcal{A},|V|}[u]=G(x,v)&\text{in~$\omega$,}\\
	u=f& \text{in the trace sense on~$\partial\omega$.}
	\end{cases}\]
	Invoking Lemma \ref{5lemma}, we get~$v\leq T(v)$ a.e. in~$\omega$. Furthermore, $$Q'_{p,\mathcal{A},V}[T(v)]=g+2V^{-}\left(|v|^{p-2}v-|T(v)|^{p-2}T(v)\right)\leq g\; \mbox{ in } \omega.$$
	
	Analogously, for any supersolution $v\in W^{1,p}(\omega)\cap C(\bar{\omega})$ of the boundary value problem
	\[\begin{cases}
	
	Q'_{p,\mathcal{A},V}[u]=g&\text{in~$\omega$,}\\
	
	u=f& \text{on~$\partial\omega$,}\\
	\end{cases}\]
	~$T(v)$ is a supersolution of the same problem with~$v\geq T(v)$ a.e. in~$\omega$.
	
	We define two sequences by recursion: for any~$k\in\mathbb{N}$,
	$$\underline{u}_{0}\triangleq \psi,\; \underline{u}_{k}\triangleq T(\underline{u}_{k-1})=T^{(k)}(\psi),
	\quad \mbox{ and } \; \bar{u}_{0} \triangleq \varphi,\; \bar{u}_{k}\triangleq T(\bar{u}_{k-1})=T^{(k)}(\varphi).$$
	Then the monotone sequences~$\{\underline{u}_{k}\}_{k\in\mathbb{N}}$ and~$\{\bar{u}_{k}\}_{k\in\mathbb{N}}$ converge  to $\underline{u}$ and~$\bar{u}$ a.e. in~$\omega$, respectively. Using \cite[Theorem 1.9]{Lieb}, we conclude that the convergence is also in~$L^{p}(\omega)$. Arguing as in the Harnack convergence principle, we may assert that~$\underline{u}$ and~$\bar{u}$ are both fixed points of~$T$ and solutions of \[\begin{cases}
	Q'_{p,\mathcal{A},V}[u]=g&\text{in~$\omega$,}\\
	u=f& \text{on~$\partial\omega$,}
	\end{cases}\]
	with~$\psi\leq\underline{u}\leq \bar{u}\leq\varphi$ in~$\omega$. The uniqueness is derived from Lemma \ref{newDiaz}.
\end{proof}
\subsubsection{Weak comparison principle}
The following weak comparison principle extends \cite[Theorem 5.3]{Pinchover} to our setting and has a similar proof to \cite[Theorem 5.3]{Pinchover}, and therefore it is omitted.
\begin{Thm}[Weak comparison principle]\label{thm_wcp}
	Let~$\omega\Subset\Omega$ be a Lipschitz domain, let~$\mathcal{A}$ satisfy Assumption~\ref{ass8}, and let $g,V\in M^{q}(p;\omega)$ with~$g\geq 0$ a.e. in~$\omega$. Assume that~$\lambda_{1}>0$. If~$u_{2}\in W^{1,p}(\omega)\cap C(\bar{\omega})$ satisfies
	\[\begin{cases}
	
	Q'_{p,\mathcal{A},V}[u_{2}]=g&\text{in~$\omega$,}\\
	
	u_{2}>0& \text{on~$\partial\omega$,}\\
	\end{cases}\]
	and $u_{1}\in W^{1,p}(\omega)\cap C(\bar{\omega})$ satisfies
	\[\begin{cases}
	
	Q'_{p,\mathcal{A},V}[u_{1}]\leq Q'_{p,\mathcal{A},V}[u_{2}]&\text{in~$\omega$,}\\
	
	u_{1}\leq u_{2}& \text{on~$\partial\omega$,}\\
	\end{cases}\]
	then $u_{1}\leq u_{2}$ in $\omega$.
\end{Thm}
\section{Agmon-Allegretto-Piepenbrink (AAP)  theorem}\label{AP}
In this section, we establish an Agmon-Allegretto-Piepenbrink (in short, AAP) type theorem. See \cite{Agmon, Allegretto1974}, \cite{Pinchover}, and \cite{Keller, AAPform}, respectively, for the counterparts, in the linear case, the quasilinear case, and the cases of graphs and certain Dirichlet forms.
\subsection{Divergence-type equation}
We introduce a divergence-type equation of the first order. For a related study, see \cite{firstreference}.
\begin{Def}\label{def2}
{\em Let~$\mathcal{A}$ satisfy Assumption~\ref{ass8} and let~$V\in M^{q}_{\loc}(p;\Omega)$. A vector field
$S\in L^{p}_{\mathrm{loc}}(\Omega;\mathbb{R}^{n})$ is a {\em solution} of the first order partial differential equation
  \begin{equation}\label{first}
  -\dive{\mathcal{A}(x,S)}+(1-p)\mathcal{A}(x,S)\cdot S+V=0 \qquad \mbox{ in } \Omega,
\end{equation}
   if
   $$\int_{\Omega}\mathcal{A}(x,S)\cdot\nabla \psi\dx+(1-p)\int_{\Omega}(\mathcal{A}(x,S)\cdot S)\psi\dx + \int_{\Omega}V\psi\dx=0,$$ for every function~$\psi\in C_{c}^{\infty}(\Omega)$, and a {\em supersolution} of the same equation
   if$$\int_{\Omega}\mathcal{A}(x,S)\cdot\nabla \psi\dx+(1-p)\int_{\Omega}(\mathcal{A}(x,S)\cdot S)\psi\dx+ \int_{\Omega}V\psi\dx\geq 0,$$ for every nonnegative function~$\psi\in C_{c}^{\infty}(\Omega)$.
   }
\end{Def}
We state a straightforward assertion without proof.
   \begin{assertion}
  All the integrals in Definition~\ref{def2} are finite.
Moreover, for any solution $S$ defined as above, the corresponding integral equality holds for all bounded functions in $W^{1,p}_{c}(\Omega)$, and for any supersolution~$S$, the corresponding integral inequality holds for all nonnegative bounded functions in $W^{1,p}_{c}(\Omega)$.
\end{assertion}
\subsection{AAP type theorem}
Following the approach in  \cite{Pinchover}, we establish the AAP type theorem. In other words, we prove that the nonnegativity of the functional~$Q_{p,\mathcal{A},V}$ on $\core$ is equivalent to the existence of a positive solution or positive supersolution of the equation~$Q'_{p,\mathcal{A},V}[u]=0$ in $\Gw$. We also obtain certain other conclusions involving the first-order equation \eqref{first} defined above. Recall that for every~$\vgf \in \core$,
	$$Q_{p,\mathcal{A},V}[\vgf]=\int_{\Omega}\left(\vert\nabla \vgf\vert_{\mathcal{A}}^{p}+V\vert \vgf\vert^{p}\right)\dx.$$
	The functional~$Q_{p,\mathcal{A},V}$ is said to be {\em nonnegative} in $\Gw$ if $Q_{p,\mathcal{A},V}[\varphi]\geq 0$ for all~$\varphi\in C^{\infty}_{c}(\Omega)$.
\begin{theorem}[AAP type theorem]\label{thm_AAP}
Let the operator~$\mathcal{A}$ satisfy Assumption~\ref{ass8}, and let $V\in M^{q}_{\loc}(p;\Omega).$ Then the following assertions are equivalent:
\begin{enumerate}
  \item[$(1)$] the functional~$Q_{p,\mathcal{A},V}$ is nonnegative in $\Gw$;
  \item[$(2)$] the equation~$Q'_{p,\mathcal{A},V}[u]= 0$ admits a positive solution~$v\in W^{1,p}_{\loc}(\Omega)$;
  \item[$(3)$] the equation~$Q'_{p,\mathcal{A},V}[u]= 0$ admits a positive supersolution~$\tilde{v}\in W^{1,p}_{\loc}(\Omega)$;
  \item[$(4)$]the first-order equation \eqref{first} admits a solution~$S\in L^{p}_{\loc}(\Omega;\mathbb{R}^{n})$;
  \item[$(5)$] the first-order equation \eqref{first} admits a supersolution $\tilde{S}\in L^{p}_{\loc}(\Omega;\mathbb{R}^{n})$.
\end{enumerate}
\end{theorem}
\begin{proof}[Proof of Theorem~\ref{thm_AAP}]
The proof of the theorem is similar to that of \cite[Theorem 4.3]{Pinchover}, but the arguments for the implications~$(2)\Rightarrow (4)$ and~$(3)\Rightarrow (5)$ are simpler.

It suffices to show~$(1)\Rightarrow (2)\Rightarrow (j)\Rightarrow (5)$, where~$j=3,4$, $(3)\Rightarrow (1)$, and~$(5)\Rightarrow (1)$.

$(1)\Rightarrow (2)$ 
Fix a point~$x_{0}\in \Omega$ and let~$\{\omega_{i}\}_{i\in \mathbb{N}}$  be a Lipschitz exhaustion of $\Gw$ such that~$x_{0}\in \omega_{1}$. Assertion (1) yields for~$i\in \mathbb{N}$,$$\lambda_{1}\big(Q_{p,\mathcal{A},V+1/i};\omega_{i}\big)\triangleq\inf_{\substack{u\in W^{1,p}_{0}(\omega_{i})\setminus\{0\}}}\frac{Q_{p,\mathcal{A},V+1/i}[u;\omega_{i}]}{\Vert u\Vert^{p}_{L^{p}(\omega_{i})}}\geq \frac{1}{i},$$ which, combined with Theorem \ref{maximum}, gives a positive solution~$v_{i}\in W^{1,p}_{0}(\omega_{i})$ of the problem~$Q'_{p,\mathcal{A},V+1/i}[u]=f_{i}$ in~$\omega_{i}$ with~$u=0$ on~$\partial\omega_{i}$, where~$f_{i}\in C^{\infty}_{c}(\omega_{i}\setminus\bar{\omega}_{i-1})\setminus\{0\}, i\geq 2,$ are nonnegative. and$$Q_{p,\mathcal{A},V+1/i}[u;\omega_{i}]\triangleq\int_{\omega_{i}}(\vert\nabla u\vert_{\mathcal{A}}^{p}+(V+1/i)\vert u\vert^{p})\dx.$$
Setting~$\omega'_{i}=\omega_{i-1}$, we get for all~$u\in W^{1,p}_{c}(\omega'_{i})$,
$$\int_{\omega_{i}'}\mathcal{A}(x,\nabla v_{i})\cdot \nabla u\dx+\int_{\omega_{i}'}\Big(V+\frac{1}{i}\Big)v_{i}^{p-1}u\dx=0.$$
Normalize~$f_{i}$ so that~$v_{i}(x_{0})=1$ for all~$i\geq 2$. Applying the Harnack convergence principle with~$\mathcal{V}_{i}\triangleq V+1/i$, we get the second assertion.

$(2)\Rightarrow (3)$ We may choose~$\tilde{v}=v$.

$(3)\Rightarrow (1)$ Let~$\tilde{v}$ be a positive supersolution of~$Q'_{p,\mathcal{A},V}[u]=0$ and~$T\triangleq -|\nabla \tilde{v}/\tilde{v}|_{\mathcal{A}}^{p-2}.$ For any~$\psi\in C^{\infty}_{c}(\Omega)$, picking~$|\psi|^{p}\tilde{v}^{1-p}\in W^{1,p}_{c}(\Omega)$ as a test function, we obtain:
 $$(p-1)\int_{\Omega}|T|_{\mathcal{A}}^{p'}|\psi|^{p}\dx\leq p\int_{\Omega}|T|_{\mathcal{A}}|\psi|^{p-1}|\nabla \psi|_{\mathcal{A}}\dx+\int_{\Omega}V|\psi|^{p}\dx.$$ Then Young's inequality~$pab\leq (p-1)a^{p'}+b^{p}$ with~$a=|T|_{\mathcal{A}}|\psi|^{p-1}$ and~$b=|\nabla \psi|_{\mathcal{A}}$ yields the first assertion. For an alternative proof, see Lemma \ref{lem_alter}. 

$(2)\Rightarrow (4)$ Let~$v$ be a positive solution of~$Q'_{p,\mathcal{A},V}[u]= 0$. Then~$1/v\in L^{\infty}_{\loc}(\Omega)$ by the weak Harnack inequality (or by the Harnack inequality if~$p>n$). Let $S\triangleq \nabla v/v$. Because~$v\in W^{1,p}_{\loc}(\Omega)$ and~$1/v\in L^{\infty}_{\loc}(\Omega)$, it follows that~$S\in L^{p}_{\loc}(\Omega;\mathbb{R}^{n})$.

We claim that~$S$ is a solution of the equation \eqref{first}. For any~$\psi\in C^{\infty}_{c}(\Omega)$, we employ~$\psi v^{1-p}\in W^{1,p}_{c}(\Omega)$, with$$\nabla\big(\psi v^{1-p}\big)=v^{1-p}\nabla \psi+(1-p)\psi v^{-p}\nabla v,$$ as a test function in the definition of the equation~$Q'_{p,\mathcal{A},V}[w]= 0$ with~$v$ as a solution.
\begin{eqnarray*}
  &&\int_{\Omega}\mathcal{A}(x,\nabla v)\!\cdot \!  v^{1-p}\nabla \psi \!\dx
  +\int_{\Omega}\mathcal{A}(x,\nabla v) \! \cdot \! (1-p)\psi v^{-p}\nabla v\dx+\int_{\Omega}V|v|^{p-2}v\psi v^{1-p}\dx\\
 &=&\int_{\Omega}\mathcal{A}\left(x,\frac{\nabla v}{v}\right)\cdot\nabla \psi\dx+(1-p)\int_{\Omega}\psi\mathcal{A}\left(x,\frac{\nabla v}{v}\right)\cdot\frac{\nabla v}{v}\dx+\int_{\Omega}V\psi\dx\\
  &=&\int_{\Omega}\mathcal{A}(x,S)\cdot\nabla \psi\dx+(1-p)\int_{\Omega}\psi\mathcal{A}(x,S)\cdot S\dx+\int_{\Omega}V\psi\dx=0.
 \end{eqnarray*}

$(3)\Rightarrow (5)$ For a positive supersolution~$\tilde{v}$ of~$Q'_{p,\mathcal{A},V}[u]= 0$ , we adopt the same argument as above with~$S$ replaced by~$\tilde{S}\triangleq \nabla \tilde{v}/\tilde{v}$, except using nonnegative test functions~$\psi\in C^{\infty}_{c}(\Omega)$ and in the last step.

$(4)\Rightarrow (5)$ We may choose~$ \tilde{S}=S$.

$(5)\Rightarrow (1)$
For any~$\psi\in C_{c}^{\infty}(\Omega)$, we get
\begin{eqnarray}
 \notag\int_{\Omega}\!\!\mathcal{A}(x,\tilde{S})\! \cdot \! \nabla |\psi|^{p}\!\!\dx
&\!\!=\!\!& p\!\int_{\Omega}\!|\psi|^{p-1}\mathcal{A}(x,\tilde{S}) \!\cdot \! \nabla |\psi| \!\dx\\\notag
&\!\!\leq \!\!& p\!\int_{\Omega} \! |\psi|^{p-1}|\tilde{S}|_{\mathcal{A}}^{p-1}|\nabla \psi|_{\mathcal{A}} \!\dx\\
&\!\leq \!&  (p\!-\! 1)\int_{\Omega}\!\!|\psi|^{p}|\tilde{S}|_{\mathcal{A}}^{p}\!\dx \!+ \! \int_{\Omega}\!|\nabla \psi|_{\mathcal{A}}^{p} \! \dx, \label{eqS}
\end{eqnarray}
where the first inequality is derived from the generalized H\"older inequality (Lemma \ref{ass1}), and in the last step, Young's inequality $pab\leq (p-1)a^{p'}+b^{p}$ is applied with~$a=|\psi|^{p-1}|\tilde{S}|_{\mathcal{A}}^{p-1}$ and~$b=|\nabla \psi|_{\mathcal{A}}$.

Because $\tilde{S}$ is a supersolution of \eqref{first}, we have
 $$\int_{\Omega}\mathcal{A}(x,\tilde{S})\cdot\nabla |\psi|^{p}\dx+(1-p)\int_{\Omega}|\tilde{S}|_{\mathcal{A}}^{p}|\psi|^{p}\dx+ \int_{\Omega}V|\psi|^{p}\dx\geq 0, $$
which together with \eqref{eqS} implies $Q_{p,\mathcal{A},V}[\psi]\geq 0$ for all $\psi\in C_{c}^{\infty}(\Omega)$.
\end{proof}
\begin{corollary}\label{newuniqueness}
  Let $\omega\Subset\Omega$ be a domain, let $\mathcal{A}$ satisfy Assumption~\ref{ass8}, and let~$V\in M^{q}(p;\omega)$. Then the principal eigenvalue is unique.
\end{corollary}
\begin{proof}
Let~$\lambda$ be any eigenvalue with an eigenfunction~$v_{\lambda}\geq 0$. By Harnack's inequality, the eigenfunction~$v_{\lambda}$ is positive in~$\omega$. Then~$v_{\lambda}$ is a positive solution of~$Q'_{p,\mathcal{A},V-\gl}[u]=0$. By the AAP type theorem, the functional~$Q_{p,\mathcal{A},V-\gl}$ is nonnegative in~$\omega$ and hence by the definition of~$\gl_{1}$ and Lemma \ref{easylemma}, we get $\gl_{1}\leq\gl\leq \gl_{1}$. Thus, $\lambda_{1}=\gl$.
\end{proof}
\section{Criticality theory}\label{criticality}
In this section we introduce the notions of criticality and subcriticality and establish fundamental results in criticality theory for the functional $Q_{p,\mathcal{A},V}$ that generalize the counterpart results in \cite[Section 4B]{Pinchover} and \cite[Theorem~6.8]{Regev}.

{First, let us discuss the following extra assumption for~$1<p<2$ that we need for certain results in the sequel.}
\begin{assumption}\label{ngradb}
	\emph{{Let $1<p<2$ and $V\in M^{q}_{\loc}(p;\Omega)$. On top of Assumption~\ref{ass8} and  Assumption \ref{ass2}, we assume that $\mathcal{A}$ and $V$ are regular enough to guarantee that positive solutions of the underlying equations are in $W^{1,\infty}_{\loc}$. Moreover, in this case we consider only positive supersolutions in $W^{1,\infty}_{\loc}$ (cf.  \cite{Pinchover})}} \end{assumption}

\begin{remark}
	{\emph{For sufficient conditions to ensure the above assumption see \cite{Lieb1}. For a more concrete example, see the hypothesis (H1) in \cite{Pinchover}. Another example  is given by the pseudo $p$-Laplacian plus a locally bounded potential by \cite[page 805]{Lieberman}.}}
\end{remark}
\subsection{Characterizations of criticality}\label{subsec_null}
\subsubsection{Null-sequences and ground states}
\begin{Def}{\em
Let~$\mathcal{A}$ satisfy Assumption~\ref{ass8} and let $V\in M^{q}_{\loc}(p;\Omega)$.
  \begin{itemize}
    \item If there exists a nonzero nonnegative function~$W\!\in\! M^{q}_{\loc}(p;\Omega)\!\setminus\!\{0\}$  such that
    \begin{equation*}\label{subcritical}
    Q_{p,\mathcal{A},V}[\varphi]\geq \int_{\Omega}W|\varphi|^{p}\dx,
    \end{equation*}
    for all $\varphi\in C^{\infty}_{c}(\Omega)$, we say that the functional $Q_{p,\mathcal{A},V}$ is \emph{subcritical} in $\Gw$, and $W$ is a \emph{Hardy-weight} for $Q_{p,\mathcal{A},V}$ in $\Gw$.
    {Clearly, if $W$ is a Hardy-weight for $Q_{p,\mathcal{A},V}$ in $\Gw$, then there exists a Hardy-weight $0\lneq\tilde W\leq W$ 
   such that  $\tilde W\!\in\! L^{\infty}_{c}(\Omega)$ for $Q_{p,\mathcal{A},V}$ in $\Gw$.}
    
    \item if $Q_{p,\mathcal{A},V}$ is nonnegative in~$\Omega$ but $Q_{p,\mathcal{A},V}$ does not admit a Hardy-weight in $\Gw$, we say that the functional $Q_{p,\mathcal{A},V}$ is \emph{critical} in $\Gw$;
    \item if there exists~$\varphi\in C^{\infty}_{c}(\Omega)$ such that $Q_{p,\mathcal{A},V}[\varphi]<0$, we say that the functional~$Q_{p,\mathcal{A},V}$ is \emph{supercritical} in $\Gw$.
  \end{itemize}

}
\end{Def}
\begin{Def}
   \emph{ Let~$\mathcal{A}$ satisfy Assumption~\ref{ass8} and let $V\in M^{q}_{\loc}(p;\Omega)$. A nonnegative sequence~$\{u_{k}\}_{k\in\mathbb{N}}\subseteq W^{1,p}_{c}(\Omega)$ is called a \emph{null-sequence} with respect to the nonnegative functional~$Q_{p,\mathcal{A},V}$ in~$\Omega$ if
    \begin{itemize}
      \item for every $k\in\mathbb{N}$, the function~$u_{k}$ is bounded in~$\Omega$;
      \item there exists a fixed open set~$U\Subset\Omega$ such that~$\Vert u_{k}\Vert_{L^{p}(U)}=1$ for all~$k\in\mathbb{N}$;
      \item and~$\displaystyle{\lim_{k\rightarrow\infty}}Q_{p,\mathcal{A},V}[u_{k}]=0.$
    \end{itemize}}
  \end{Def}
  \begin{Def}
  \emph{  Let~$\mathcal{A}$ satisfy Assumption~\ref{ass8} and let $V\in M^{q}_{\loc}(p;\Omega)$. A \emph{ground state} of the nonnegative functional~$Q_{p,\mathcal{A},V}$ is a positive  function~$\phi\in W^{1,p}_{\loc}(\Omega)$, which is an~$L^{p}_{\loc}(\Omega)$ limit of a null-sequence.}
  \end{Def}
\begin{lem}\label{simplelemma}
		Let~$\mathcal{A}$ satisfy Assumption~\ref{ass8} and let $V\in M^{q}_{\loc}(p;\Omega)$. If a nonnegative sequence~$\{u_{k}\}_{k\in\mathbb{N}}\subseteq W^{1,p}_{c}(\Omega)$ satisfies:
		\begin{itemize}
            \item for every $k\in\mathbb{N}$, the function~$u_{k}$ is bounded in~$\Omega$;
			\item there exists a fixed open set~$U\Subset\Omega$ such that~$\Vert u_{k}\Vert_{L^{p}(U)}\asymp 1$ for all~$k\in\mathbb{N}$;
			\item $\displaystyle{\lim_{k\rightarrow\infty}}Q_{p,\mathcal{A},V}[u_{k}]=0;$
			\item and~$\{u_{k}\}_{k\in\mathbb{N}}$ converges in~$L^{p}_{\loc}(\Omega)$ to a positive~$u\in W^{1,p}_{\loc}(\Omega)$,
		\end{itemize}
		then $u$ is a ground state up to a multiplicative constant.
	\end{lem}
	\begin{proof}
		By the second condition, we may assume that up to a subsequence $\displaystyle{\lim_{k\rightarrow\infty}}\Vert u_{k}\Vert_{L^{p}(U)}=C_{0}$ for some positive constant~$C_{0}$. Set $C_{k}\triangleq\Vert u_{k}\Vert_{L^{p}(U)}$. Then $\left\{u_{k}/C_{k}\right\}_{k\in\mathbb{N}}$ is a null-sequence converging in~$L^{p}_{\loc}(\Omega)$  to~${u/C_{0}}$.
\end{proof}
  \begin{corollary}\label{nullrem}
    Let $\omega\Subset\Omega$ be a domain, let $\mathcal{A}$ satisfy Assumption~\ref{ass8} and  let $V\in M^{q}_{\loc}(p;\Omega)$. Then a positive principal eigenfunction~$v$ associated to the principal eigenvalue $\gl_1=\gl_1(Q_{p,\mathcal{A},V};\omega)$ is a ground state of the functional~$Q_{p,\mathcal{A},V-\lambda_{1}}$ in~$\omega$.
  \end{corollary}
  \begin{proof}
    Let $v'\in W^{1,p}_{0}(\omega)$ be a principal eigenfunction associated to $\gl_1$, and let $\{\varphi_{k}\}_{k\in\mathbb{N}}\subseteq C^{\infty}_{c}(\omega)$
   be a sequence approximating $v'$ in $W^{1,p}_{0}(\omega)$. Then Lemma~\ref{functionalcv} implies that
   	$$\lim_{k\rightarrow\infty}Q_{p,\mathcal{A},V-\lambda_{1}}[\varphi_{k};\omega]=Q_{p,\mathcal{A},V-\lambda_{1}}[v';\omega]=0,\; \mbox{and } \; \Vert \varphi_{k}\Vert_{L^{p}(U)}\asymp 1 \;\; \forall k\in\mathbb{N},$$
   	 where~$U\Subset\omega$ is a fixed nonempty open set. By Lemma \ref{simplelemma}, for some positive constant~$C$, the principal eigenfunction~$v\triangleq Cv'$ is a ground state of~$Q_{p,\mathcal{A},V-\lambda_{1}}$ in~$\omega$.
  \end{proof}
   \begin{proposition}\label{mainlemma}
    Let~$\{u_{k}\}_{k\in\mathbb{N}}$ be a null-sequence with respect to a nonnegative functional $Q_{p,\mathcal{A},V}$ in~$\Omega$,  where $\mathcal{A}$ satisfies assumptions~\ref{ass8} and \ref{ass2}, and  $V\in M^{q}_{\loc}(p;\Omega)$. Let $v$ be a  positive supersolution {(and assume $v\in W^{1,\infty}_{\loc}(\Omega)$ if~$p<2$)} of $Q'_{p,\mathcal{A},V}[u]=0$  in $\Omega$ and let $w_{k}=u_{k}/v$, where $k\in\mathbb{N}$. Then the sequence~$\{w_{k}\}_{k\in\mathbb{N}}$ is bounded in~$W^{1,p}_{\loc}(\Omega)$ and~$\nabla w_{k}\rightarrow 0$ as~$k\rightarrow\infty$ in~$L^{p}_{\loc}(\Omega;\mathbb{R}^{n})$.
  \end{proposition}
  \begin{proof}
    Let $U$ be a fixed open set such that for all~$k\in\mathbb{N}$,~$\Vert u_{k}\Vert_{L^{p}(U)}=1$, and let  $U\Subset\omega\Subset\Omega$ be  a Lipschitz domain. Using the Minkowski inequality, the Poincar\'{e} inequality, and the weak Harnack inequality, we obtain for every $k\in\mathbb{N}$,
    \begin{eqnarray*}
      \Vert w_{k}\Vert_{L^{p}(\omega)}&\leq& \Vert w_{k}-\langle w_{k}\rangle_{U}\Vert_{L^{p}(\omega)}+\langle w_{k}\rangle_{U}\left(\mathcal{L}^{n}(\omega)\right)^{1/p}\\
      &\leq& C(n,p,\omega,U)\Vert \nabla w_{k}\Vert_{L^{p}(\omega;\mathbb{R}^{n})}+\frac{1}{\inf_{U}v}\langle u_k\rangle_{U}\left(\mathcal{L}^{n}(\omega)\right)^{1/p}.
    \end{eqnarray*}
  By H\"{o}lder's inequality, noting that $\Vert u_{k}\Vert_{L^{p}(U)}=1$, we obtain
  \begin{equation}\label{estimate}
    \Vert w_{k}\Vert_{L^{p}(\omega)}\leq C(n,p,\omega,U)\Vert \nabla w_{k}\Vert_{L^{p}(\omega;\mathbb{R}^{n})}+\frac{1}{\inf_{U}v}\left(\frac{\mathcal{L}^{n}(\omega)}{\mathcal{L}^{n}(U)}\right)^{1/p}.
  \end{equation}
   By Lemma \ref{strictconvexity} with $\xi_{1}=\nabla(vw_{k})=\nabla(u_{k})$ and $\xi_{2}=w_{k}\nabla v$, we obtain
   $$|\nabla u_{k}|_{\mathcal{A}}^{p} -w_{k}^{p}|\nabla v|_{\mathcal{A}}^{p} - p\mathcal{A}(x,w_k\nabla v)\!\cdot\! v\nabla w_{k}\geq 0,$$
    which together with {(Assumption~\ref{ass2})} implies
    \begin{eqnarray*}
     { C_\gw(p, \mathcal{A})\int_{\gw}[\nabla u_{k}, w_{k}\nabla v]_{\mathcal{A}}\dx}
      &\leq& \int_{\omega}\left(|\nabla u_{k}|_{\mathcal{A}}^{p} -w_{k}^{p}|\nabla v|_{\mathcal{A}}^{p} - p\mathcal{A}(x,w_k\nabla v)\!\cdot\! v\nabla w_{k}\right)\dx\\
      &\leq& \int_{\Omega}\left(|\nabla u_{k}|_{\mathcal{A}}^{p} - w_{k}^{p}|\nabla v|_{\mathcal{A}}^{p} - v\mathcal{A}(x,\nabla v)\!\cdot\! \nabla\left(w_{k}^{p}\right)\right)\dx\\
      &=&\int_{\Omega}|\nabla u_{k}|_{\mathcal{A}}^{p}\dx-\int_{\Omega}\mathcal{A}(x,\nabla v)\!\cdot\! \nabla\left(w_{k}^{p}v\right)\!\dx.
    \end{eqnarray*}
     Since $v$ is a positive supersolution, we obtain
    $$ {C_\gw(p, \mathcal{A})\int_{\gw}[\nabla u_{k}, w_{k}\nabla v]_{\mathcal{A}}\dx}
    \leq \int_{\Omega}|\nabla u_{k}|_{\mathcal{A}}^{p}\dx+\int_{\Omega}Vu_{k}^{p}\dx=Q_{p,\mathcal{A},V}[u_{k}].$$
   {Suppose that~$p\geq 2$.}
    By the  weak Harnack inequality, and the ellipticity condition \eqref{structure}, we get for a positive constant~$c$ which does not depend on~$k$,
    	$${c\int_{\omega}|\nabla w_{k}|^{p}\dx\leq C_\gw(p, \mathcal{A})\int_{\Omega}v^{p}|\nabla w_{k}|^{p}_{\mathcal{A}}\dx}\leq Q_{p,\mathcal{A},V}[u_{k}]\rightarrow 0\; \mbox{ as } k\to \infty.$$
    	Consequently,  {by H\"{o}lder's inequality,}
    	$$\nabla w_{k}\rightarrow 0\; \mbox{ as } k\to \infty\; \mbox{ in } L^{p}_{\loc}(\Omega;\mathbb{R}^{n}),$$ and  this and \eqref{estimate} also imply that~$\{w_{k}\}_{k\in\mathbb{N}}$ is bounded in~$W^{1,p}_{\loc}(\Omega)$.
{The proof for~$p<2$ is similar to the counterpart proof of \cite[Proposition~4.11]{Pinchover}.
}
  \end{proof}
\begin{Thm}\label{mainthm}   Let~$\mathcal{A}$  satisfy assumptions~\ref{ass8} and \ref{ass2} and let $V\in M^{q}_{\loc}(p;\Omega)$. {In addition, for $1<p<2$ suppose that Assumption~\ref{ngradb} holds.} Assume that the functional $Q_{p,\mathcal{A},V}$ is nonnegative  on $\core$. Then every null-sequence of~$Q_{p,\mathcal{A},V}$ converges, both in~$L^{p}_{\loc}(\Omega)$ and a.e. in~$\Omega$, to a unique (up to a multiplicative constant) positive supersolution  of the equation $Q'_{p,\mathcal{A},V}[u]=0$ in~$\Omega$.

 Furthermore, a ground state is in~$C^{\gamma}_{\loc}(\Omega)$ for some~$0<\gamma\leq 1$, and it is the unique positive
  solution and the unique positive  supersolution of $Q'_{p,\mathcal{A},V}[u]=0$ in~$\Omega$.
  \end{Thm}
  \begin{Rem}
    \emph{By uniqueness we mean uniqueness up to a multiplicative constant.}
  \end{Rem}
   \begin{proof}[Proof of Theorem~\ref{mainthm}]
   By the AAP type theorem, there exists a positive
     supersolution $v\in W^{1,p}_{\loc}(\Omega)$  and a positive
      solution~$\tilde{v}\in W^{1,p}_{\loc}(\Omega)$  of $Q'_{p,\mathcal{A},V}[u]=0$ in $\Gw$. Let $\{u_{k}\}_{k\in\mathbb{N}}$ be a null-sequence of $Q_{p,\mathcal{A},V}$ in $\Gw$, and set $w_{k}=u_{k}/v$. Using Proposition \ref{mainlemma}, we obtain$$\nabla w_{k}\rightarrow 0\; \mbox{ as } k\rightarrow \infty \mbox{ in } L^{p}_{\loc}(\Omega;\mathbb{R}^{n}).$$ The Rellich-Kondrachov theorem yields a subsequence, which is still denoted by~$w_{k}$, with~$w_{k}\rightarrow c$ for some~$c\geq 0$ in~$W^{1,p}_{\loc}(\Omega)$ as~$k\rightarrow\infty$.
      Since $v$ is locally bounded away from zero, it follows that up to a subsequence, $u_{k}\rightarrow cv$ a.e. in~$\Omega$ and also in~$L^{p}_{\loc}(\Omega)$. 
      Therefore, $c=1/\Vert v\Vert_{L^{p}(U)}>0$. Furthermore, any null-sequence~$\{u_{k}\}_{k\in\mathbb{N}}$ converges (up to a positive constant multiple) to the same positive supersolution~$v$. Noting that the solution~$\tilde{v}$  is
      also  a positive supersolution, we conclude that~$v=C\tilde{v}$ for some~$C>0$. It follows that~$v$ is also the unique positive solution.
  \end{proof}
  As a corollary of the above theorem we have:
     \begin{Thm}\label{complement}
     Let $\omega\Subset\Omega$ be a domain. Let $\mathcal{A}$ satisfy assumptions~\ref{ass8} and \ref{ass2} and let~$V\in M^{q}(p;\omega)$. In addition, for $1<p<2$ suppose that Assumption~\ref{ngradb} holds. Suppose that the equation $Q'_{p,\mathcal{A},V}[u]=0$ in $\gw$ admits a proper positive supersolution in~$W^{1,p}(\omega)$. Then the principal eigenvalue $\lambda_{1} =\lambda_{1}(Q_{p,\mathcal{A},V};\omega)$ is strictly positive.

     Therefore, all the assertions in Theorem \ref{maximum} are equivalent if~$\mathcal{A}$ and~$V$ are as above and~$\omega\Subset\Omega$ is a Lipschitz domain.
   \end{Thm}
   \begin{proof}
   We need to prove $(4')\Rightarrow (3)$ in Theorem \ref{maximum}.  Indeed, by the AAP type theorem,~$Q_{p,\mathcal{A},V}$ is nonnegative. In particular, $\lambda_{1}\geq 0.$ If~$\lambda_{1}=0$, then positive principal eigenfunctions
    are all positive solutions  of the equation $Q'_{p,\mathcal{A},V}[u]=0$. By Corollary \ref{nullrem} and Theorem \ref{mainthm}, the positive principal eigenfunctions are the unique
     positive supersolution of $Q'_{p,\mathcal{A},V}[u]=0$ in $\gw$, but this contradicts our assumption that the equation $Q'_{p,\mathcal{A},V}[u]=0$ in $\gw$ admits a  proper positive supersolution in~$W^{1,p}(\omega)$. Hence, $\lambda_{1}>0$.
   \end{proof}
   \subsubsection{Characterizations of criticality}
   The next theorem contains fundamental characterizations of criticality or subcriticality.
   \begin{Thm}\label{thm_Poincare}
 Let~$\mathcal{A}$ satisfy assumptions~\ref{ass8} and \ref{ass2} and let~$V\in M^{q}_{\loc}(p;\Omega)$. {In addition, for $1<p<2$ suppose that Assumption~\ref{ngradb} holds.}
Assume that~$Q_{p,\mathcal{A},V}$ is nonnegative in~$\Omega$.
 Then the following assertions hold true.
 \begin{enumerate}
   \item[$(1)$] The functional~$Q_{p,\mathcal{A},V}$ is critical in~$\Omega$ if and only if~$Q_{p,\mathcal{A},V}$ has a null-sequence in~$\core$.
   \item[$(2)$] The functional~$Q_{p,\mathcal{A},V}$ has a null-sequence if and only if the equation $Q'_{p,\mathcal{A},V}[u]=0$ has a unique positive supersolution. \item[$(3)$]The functional $Q_{p,\mathcal{A},V}$  is subcritical in~$\Omega$ if and only if $Q_{p,\mathcal{A},V}$ admits a strictly positive continuous Hardy-weight $W$ in~$\Omega$.

   \item[$(4)$] Assume that~$Q_{p,\mathcal{A},V}$ admits a ground state $\phi$ in $\Gw$. Then  there exists a strictly positive continuous
    function~$W$ such that
   for $\psi\in C^{\infty}_{c}(\Omega)$ with~$\int_{\Omega}\phi\psi\dx\neq 0$ there exists a constant $C>0$ such that the following Poincar\'{e}-type inequality holds:
   $$Q_{p,\mathcal{A},V}[\varphi]+C\left\vert\int_{\Omega}\varphi\psi\dx\right\vert^{p}\geq \frac{1}{C}\int_{\Omega}W|\varphi|^{p}\dx\qquad \forall \varphi\in C^{\infty}_{c}(\Omega).$$
 \end{enumerate}
  \end{Thm}
 \begin{proof}
  $(1)$ Suppose that~$Q_{p,\mathcal{A},V}$ is critical. For every nonempty open~$U\Subset\Omega$, let$$c_{U}\triangleq\inf_{\substack{0\leq \varphi\in C^{\infty}_{c}(\Omega)\\\Vert \varphi\Vert_{L^{p}(U)}=1}}Q_{p,\mathcal{A},V}[\varphi]=\inf_{\substack{\varphi\in C^{\infty}_{c}(\Omega)\\\Vert \varphi\Vert_{L^{p}(U)}=1}}Q_{p,\mathcal{A},V}[\varphi],$$ where the equality is an immediate corollary of Lemma \ref{functionalcv}. Pick $W\in C^{\infty}_{c}(U)\setminus\{0\}$ such that~$0\leq W\leq 1$. Then for all~$\varphi\in C^{\infty}_{c}(\Omega)$ with~$\Vert \varphi\Vert_{L^{p}(U)}=1$, we have
  $$c_{U}\int_{\Omega}W|\varphi|^{p}\dx\leq c_{U}\leq Q_{p,\mathcal{A},V}[\varphi].$$ Because  $Q_{p,\mathcal{A},V}$ is critical in $\Gw$,
   it follows that $c_{U}=0$. Hence, a minimizing sequence of the above variational problem is a null-sequence of $Q_{p,\mathcal{A},V}$ in $\Gw$.

  Suppose that $Q_{p,\mathcal{A},V}$ admits a null-sequence in $\Gw$. By Theorem \ref{mainthm}, we get a positive solution~$v$  of~$Q'_{p,\mathcal{A},V}[u]=0.$  If $Q_{p,\mathcal{A},V}$ is subcritical in $\Omega$ with a nontrivial nonnegative Hardy-weight $W$,  then the  AAP type theorem gives a positive solution $w$  of the equation
  $Q'_{p,\mathcal{A},V-W}[u]=0$ in $\Gw$. The function $w$ is also a proper positive supersolution
   of the equation $Q'_{p,\mathcal{A},V}[u]=0$. Therefore, $w$ and $v$ are two positive supersolutions of the above equation, but this contradicts Theorem~\ref{mainthm}.

  $(2)$ Suppose that the equation $Q'_{p,\mathcal{A},V}[u]=0$ in $\Gw$ admits a unique positive supersolution. If~$Q_{p,\mathcal{A},V}$ does not admits a null-sequence, then~$Q_{p,\mathcal{A},V}$ is subcritical by~$(1)$. The same argument as in the second part of the proof of~$(1)$ leads to a contradiction. The other direction follows from Theorem \ref{mainthm}.


  $(3)$ Suppose that~$Q_{p,\mathcal{A},V}$ is subcritical. Let~$\{U_{k}\}_{k\in\mathbb{N}}$ be an open covering of $\Omega$ with $U_k\Subset \Gw$ and $\cup_{k\in\mathbb{N}}U_{k}=\Omega$.  Let $\{\chi_{k}\}_{k\in\mathbb{N}}$ be a locally finite smooth partition of unity subordinated to the covering. Then by the proof of $(1)$, for every~$k\in \N$, we have~$c_{U_k}>0$. Then  for all $\varphi\in C^{\infty}_{c}(\Omega)$ and every~$k\in\mathbb{N}$ we have
  $$2^{-k}Q_{p,\mathcal{A},V}[\varphi]\geq 2^{-k}c_{U_k}\int_{U_{k}}|\varphi|^{p}\dx\geq 2^{-k}C_{k}\int_{\Omega}\chi_{k}|\varphi|^{p}\dx,$$
  where~$C_{k}\triangleq \min\{c_{U_{k}},1\}$ for $k\in\mathbb{N}$.
  ~
   Then  for all $\varphi\in C^{\infty}_{c}(\Omega)$ and every~$k\in\mathbb{N}$ we have
  $$2^{-k}Q_{p,\mathcal{A},V}[\varphi]\geq 2^{-k}C_{k}\int_{U_{k}}|\varphi|^{p}\dx\geq 2^{-k}C_{k}\int_{\Omega}\chi_{k}|\varphi|^{p}\dx.$$
  Adding together all the above inequalities over all~$k\in\mathbb{N}$, we get
  $$Q_{p,\mathcal{A},V}[\varphi]\geq \int_{\Omega}W|\varphi|^{p}\dx \qquad \forall \varphi\in C^{\infty}_{c}(\Omega),$$
   where $W\triangleq \sum_{k=1}^{\infty}2^{-k}C_{k}\chi_{k} >0$ is smooth (recall that this series is locally finite).
 The other direction follows from the definition of subcriticality.

 $(4)$  For every nonempty open set~$U\Subset\Omega$ and every~$\varphi\in C^{\infty}_{c}(\Omega)$, let
  $$Q_{p,\mathcal{A},V}^{U}[\varphi]\triangleq Q_{p,\mathcal{A},V}[\varphi]+\int_{U}|\varphi|^{p}\dx,$$
  which is subcritical because $Q_{p,\mathcal{A},V}$ is nonnegative. By~$(3)$, for every nonempty open set
  ~$U\Subset\Omega$, there is a positive continuous function~$W$ in~$\Omega$ such that for all~$\varphi\in C^{\infty}_{c}(\Omega)$,
  \begin{equation}\label{eq_W}
  Q_{p,\mathcal{A},V}^{U}[\varphi]\geq \int_{\Omega}W(x)|\varphi|^{p}\dx.
  \end{equation}
  Fix $\psi\in C^{\infty}_{c}(\Omega)$ with~$\int_{\Omega}\phi\psi\dx\neq 0$.
Assume that for every~$U\!\Subset\!\Omega$, there exists a nonnegative sequence $\{\varphi_{k}\}_{k\in\mathbb{N}}\subseteq C^{\infty}_{c}(\Omega)$ such that
  $$\int_{U}|\varphi_{k}|^{p}\dx=1,\quad Q_{p,\mathcal{A},V}[\varphi_{k}]\rightarrow 0,\quad\mbox{and}~\int_{\Omega}\varphi_{k}\psi\dx\rightarrow 0,\quad\mbox{as}~k\rightarrow\infty.$$ Because~$\{\varphi_{k}\}_{k\in\mathbb{N}}$ is a null-sequence, by Theorem \ref{mainthm}, ~$\{\varphi_{k}\}_{k\in\mathbb{N}}$ converges (up to a multiplicative constant) in~$L^{p}_{\loc}(\Omega)$ to the ground state. Furthermore,$$\lim_{k\rightarrow\infty}\int_{\Omega}\varphi_{k}\psi\dx=\int_{\Omega}\phi\psi\dx\neq 0,$$
  and we arrive at a contradiction.
  Therefore, there exists a nonempty open~$U\Subset\Omega$ such that for all~$\varphi\in C^{\infty}_{c}(\Omega)$ and some positive constant~$C$,
  $$\int_{U}|\varphi|^{p}\dx\leq C\Big(Q_{p,\mathcal{A},V}[\varphi]+\Big\vert\int_{\Omega}\varphi\psi\dx\Big\vert^{p}\Big).$$
  By combining the above inequality with \eqref{eq_W}, we obtain the desired inequality.
 \end{proof}
 \begin{corollary}\label{subcriticaleg}
   Let~$\mathcal{A}$ satisfy assumptions~\ref{ass8} and \ref{ass2} and let $V\in M^{q}_{\loc}(p;\Omega)$. {In addition, for $1<p<2$ suppose that Assumption~\ref{ngradb} holds.}
   Then $Q_{p,\mathcal{A},V}$ is subcritical in a domain~$\omega\Subset\Omega$ if and only if $\lambda_{1}(Q_{p,\mathcal{A},V};\omega)>0$.
 \end{corollary}
 \begin{proof}
  Suppose that $Q_{p,\mathcal{A},V}$ is subcritical in $\omega $. Therefore, it admits a Hardy-weight $W$ in~$\omega$. The AAP type theorem (Theorem~\ref{thm_AAP}) implies that there exists a positive solution $v$  of $Q_{p,\mathcal{A},V-W}'[u]=0$ in~$\omega$. Clearly,  $v$ is a proper positive supersolution
  of $Q_{p,\mathcal{A},V}'[u]=0$ in~$\omega$. By Theorem \ref{complement}, we have~$\lambda_{1}(Q_{p,\mathcal{A},V};\omega)>0$.

On the other hand, if $\lambda_{1}(Q_{p,\mathcal{A},V};\omega)>0$, then~$\lambda_{1}$ is a Hardy-weight for~$Q_{p,\mathcal{A},V}$ in~$\omega$ and hence~$Q_{p,\mathcal{A},V}$ is subcritical.
 \end{proof}
\subsection{Perturbation results and applications}\label{ssect_pert}
The present subsection is mainly intended for certain perturbation results.  Our perturbations results are divided into two cases. One is a domain perturbation, and the other concerns certain potential perturbations. As an application we show that a critical operator admits a null-sequence that converges locally uniformly to its ground state.
\subsubsection{Criticality theory under a domain perturbation}
The following is a straightforward result (see \cite[Proposition 4.2]{Tintarev}).
 \begin{proposition}\label{two}
   Let~$\mathcal{A}$ satisfy assumptions~\ref{ass8} and \ref{ass2} and let~$V\in M^{q}_{\loc}(p;\Omega)$.  {In addition, for $1<p<2$ suppose that Assumption~\ref{ngradb} holds.}
   Let $\Omega_{1}\subseteq\Omega_{2} \subseteq \Gw$ be subdomains such that $\Omega_{2}\setminus \overline{\Omega_{1}}\neq \emptyset$.
\begin{enumerate}
     \item[$(a)$] If~$Q_{p,\mathcal{A},V}$ is nonnegative in~$\Omega_{2}$, then~$Q_{p,\mathcal{A},V}$ is subcritical in~$\Omega_{1}$.
     \item[$(b)$] If~$Q_{p,\mathcal{A},V}$ is critical in~$\Omega_{1}$, then~$Q_{p,\mathcal{A},V}$ is supercritical in~$\Omega_{2}$.
   \end{enumerate}
 \end{proposition}
\begin{corollary}
    Let~$\mathcal{A}$ satisfy assumptions~\ref{ass8} and \ref{ass2}, and let~$V\in M^{q}_{\loc}(p;\Omega)$. {In addition, for $1<p<2$ suppose that Assumption~\ref{ngradb} holds.}  If~$Q_{p,\mathcal{A},V}$ is nonnegative in~$\Omega$, then for all domains~$\omega\Subset\Omega$, we have~$\lambda_{1}(Q_{p,\mathcal{A},V};\omega)>0$.
 \end{corollary}
 \begin{proof}
   The result follows directly from Proposition \ref{two} and Corollary \ref{subcriticaleg}.
 \end{proof}
 \subsubsection{Criticality theory under potential perturbations}
We state  here certain results on perturbations by a potential whose proofs are as in the proofs of \cite[Proposition 4.8]{Pinchover}, \cite[Corollary 4.17]{Pinchover},  \cite[Propositions 4.4 and 4.5]{Tintarev}.
	\begin{proposition}\label{prop1}
		Suppose that~$\mathcal{A}$ satisfies Assumption \ref{ass8}, $V_{2}\geq V_{1}$ a.e. in $\Omega$, where $V_i\in M^{q}_{\loc}(p;\Omega)$ for~$i=1,2$, and
		$\mathcal{L}^{n}(\{x\in \Gw : V_{2}(x)>V_{1}(x) \})>0$.
		\begin{enumerate}
			\item[$(1)$] If~$Q_{p,\mathcal{A},V_{1}}$ is nonnegative in~$\Omega$, then~$Q_{p,\mathcal{A},V_{2}}$ is subcritical in~$\Omega$.
			\item[$(2)$] If~$Q_{p,\mathcal{A},V_{2}}$ is critical in~$\Omega$, then~$Q_{p,\mathcal{A},V_{1}}$ is supercritical in~$\Omega$.
		\end{enumerate}
	\end{proposition}
 \begin{cor}\label{interval}
    Let~$\mathcal{A}$ satisfy assumptions~\ref{ass8} and \ref{ass2}, and let~$V_{i}\in M^{q}_{\loc}(p;\Omega)$, where $i=0,1$.  {In addition, for $1<p<2$ suppose that Assumption~\ref{ngradb} holds.}  Assume that~$Q_{p,\mathcal{A},V_{i}}$ are nonnegative for~$i=1,2$.
     Let~$V_{t}\triangleq(1-t)V_{0}+tV_{1} $ for~$t\in [0,1]$.
     Then $Q_{p,\mathcal{A},V_{t}}$ is nonnegative in~$\Omega$ for all~$t\in [0,1]$. Moreover, if $\mathcal{L}^{n}\left(\{V_{0}\neq V_{1}\}\right)\!>\!0$, then~$Q_{p,\mathcal{A},V_{t}}$ is subcritical in~$\Omega$ for every~$t\in (0,1)$.
 \end{cor}
 \begin{proposition}\label{prop_subcritical}
  Let~$\mathcal{A}$ satisfy assumptions~\ref{ass8} and \ref{ass2}, and let~$V\in M^{q}_{\loc}(p;\Omega)$. {In addition, for $1<p<2$ suppose that Assumption~\ref{ngradb} holds.}
   Assume that~$Q_{p,\mathcal{A},V}$ is subcritical in~$\Omega$ and $\mathbf{V}\in L_c^{\infty}(\Omega)\setminus\{0\}$ is such that~$\mathbf{V}\ngeq 0$. Then there is~$\tau_{+}>0$ and~$\tau_{-}\in[-\infty,0)$ such that $Q_{p,\mathcal{A},V+t\mathbf{V}}$ is subcritical in~$\Omega$ if and only if~$t\in(\tau_{-},\tau_{+})$. In addition,~$Q_{p,\mathcal{A},V+\tau_{+}\mathbf{V}}$ is critical in~$\Omega$.
\end{proposition}
\begin{proposition}
    Let~$\mathcal{A}$ satisfy assumptions~\ref{ass8} and \ref{ass2}, and let~$V\in M^{q}_{\loc}(p;\Omega)$. {In addition, for $1<p<2$ suppose that Assumption~\ref{ngradb} holds.}
     Assume that~$Q_{p,\mathcal{A},V}$ is critical in~$\Omega$ with a ground state $v$. Let $\mathbf{V}\in L_c^{\infty}(\Omega)$. Then there is~$0<\tau_{+}\leq \infty$ such that~$Q_{p,\mathcal{A},V+t\mathbf{V}}$ is subcritical in~$\Omega$ for~$t\in (0,\tau_{+})$ if and only if~$\int_{\Omega}\mathbf{V}|v|^{p}\dx>0.$
  \end{proposition}
  \subsubsection{Locally uniformly convergent null-sequence}
  The following is an important application of the above perturbation results.
  \begin{lem}\label{localuniform}
   Let~$\mathcal{A}$ satisfy assumptions~\ref{ass8} and \ref{ass2}, and let~$V\in M^{q}_{\loc}(p;\Omega)$. {In addition, for $1<p<2$ suppose that Assumption~\ref{ngradb} holds.}
  Assume that~$Q_{p,\mathcal{A},V}$ is critical in~$\Omega$. Then~$Q_{p,\mathcal{A},V}$ admits a null-sequence $\{\phi_{i}\}_{i\in\mathbb{N}}\subseteq C^{\infty}_{c}(\Omega)$ converging locally uniformly to the ground state $\phi$.
  \end{lem}
  \begin{proof}
    Let~$\{\omega_{i}\}_{i\in\mathbb{N}}$ be a Lipschitz exhaustion of~$\Omega$,~$x_{0}\in\omega_{1}$, and~$\mathbf{V}\in C^{\infty}_{c}(\Omega)\setminus\{0\}$  a nonnegative function such that~$\supp(\mathbf{V})\Subset\omega_{1}$. By virtue of Proposition~\ref{prop_subcritical}, for every~$i\in\mathbb{N}$, there exists~$t_{i}>0$ such that the functional~$Q_{p,\mathcal{A},V-t_{i}\mathbf{V}}$ is critical in~$\omega_{i}$.
Let $\phi_{i}'$ be the ground state of $Q_{p,\mathcal{A},V-t_{i}\mathbf{V}}$ in~$\omega_{i}$ satisfying $\phi_{i}'(x_0)=1$. Clearly. $\lim_{i\to\infty}t_i =0$, and $\lambda_{1}(Q_{p,\mathcal{A},V-t_{i}\mathbf{V}};\omega_{i})\!=\!0$, hence, $\phi_{i} \!\in\! W^{1,p}_{0}(\omega_{i})$, and~$Q_{p,\mathcal{A},V-t_{i}\mathbf{V}}[\phi_{i}']\!=\!0$.
By Theorems~\ref{HCP} and \ref{thm_Poincare}, it follows that the sequence~$\{\phi_{i}'\}_{i\in\mathbb{N}}$ converges locally uniformly to $c\phi$, the ground state of  $Q'_{p,\mathcal{A},V}$  in $\Omega$, where $c>0$ is a constant, and $\int_{\omega_{1}}|\phi'_i|^{p}\dx \asymp \int_{\omega_{1}}|\phi|^{p}\dx\asymp 1$.
It follows that~$\displaystyle{\lim_{i\rightarrow\infty}}Q_{p,\mathcal{A},V}[\phi_{i}'] \!=\!\displaystyle{\lim_{i\rightarrow\infty}}t_{i} \! \int_{\omega_{1}}\!\mathbf{V}(\phi_{i}')^{p}\dx = 0$.

By virtue of \cite[page 250, Theorem 1]{Evans} and \cite[page 630, Theorem 6]{Evans}, there exists a nonnegative approximating sequence~$\{\phi_{i}\}_{i\in\mathbb{N}} \! \subseteq \! C^{\infty}_{c}(\Omega)$ such that~$\displaystyle{\lim_{i\rightarrow\infty}}Q_{p,\mathcal{A},V}[\phi_{i}] \!=\!0,$ and $\{\phi_{i}\}$ converges locally uniformly to~$\phi$ in~$\Omega$. Hence, $\int_{\omega_{1}}|\phi_i|^{p}\dx \!\asymp \!1.$ By Lemma \ref{simplelemma}, the desired result follows.
    \end{proof}
\subsection{Hardy–Sobolev–Maz’ya inequality and $(\mathcal{A},V)$-capacity}\label{A,V-capacity}
The following definition of capacity is a counterpart of \cite[Definition 6.7]{Regev}.
\begin{Def}\label{AVcapacity}
\emph{  Let~$\mathcal{A}$ satisfy Assumption~\ref{ass8} and let~$V\in M^{q}_{\loc}(p;\Omega)$. Assume that the functional~$Q_{p,\mathcal{A},V}$ is nonnegative on~$C^{\infty}_{c}(\Omega)$. For every compact subset~$K$ of~$\Omega$, we define the\emph{~$(\mathcal{A},V)$-capacity} of~$K$ in~$\Omega$ as
   $$\capacity_{\mathcal{A},V}(K,\Omega)\triangleq \inf\left\{Q_{p,\mathcal{A},V}[\varphi]:\varphi\in C^{\infty}_{c}(\Omega), \varphi\geq 1 \mbox{ on } K\right\}.$$
}
\end{Def}
\begin{remark}
\emph{For the $p$-capacity and the $(p;r)$-capacity, see \cite[Chapter 2]{HKM} and \cite[Section 2.1]{Maly}. For a relationship between the~$p$-capacity and the~$p$-parabolicity in a Riemannian manifold, see \cite{parabolicity1,parabolicity2}. Recall that $|\xi|_{\mathcal{A}}^{p}=pF(x,\xi)$ for a.e.~$x\in\Omega$ and all~$\xi\in\mathbb{R}^{n}$. For the variational~$F$-capacity, which is a Choquet capacity as guaranteed by \cite[Theorem 5.31]{HKM}, we refer to \cite[Section 5.30]{HKM}. }
\end{remark}
The following theorem is an extension of \cite[Theorem 6.8]{Regev}, \cite[Theorem 4.5]{Regev20}, and \cite[Theorem 3.4]{Regev21}. The proof is omitted since it is similar to that of \cite[Theorem 4.5]{Regev20}.
\begin{Thm}\label{newthm}
  Let~$\mathcal{A}$ satisfy assumptions~\ref{ass8} and \ref{ass2} and let~$V\in M^{q}_{\loc}(p;\Omega)$. {In addition, for $1<p<2$ suppose that Assumption~\ref{ngradb} holds.} Assume that $Q_{p,\mathcal{A},V}$ is nonnegative on~$C^{\infty}_{c}(\Omega)$ {and  the equation~$Q'_{p,\mathcal{A},V}[w]=0$ in~$\Omega$ admits a positive supersolution~$v$ whose gradient is locally bounded.} Then the following assertions are equivalent.
  \begin{enumerate}
    \item[$(1)$] The functional~$Q_{p,\mathcal{A},V}$ is subcritical in~$\Omega$;
    \item[$(2)$] there exists a positive continuous function~$W^\ast$ in~$\Omega$ such that for all~$\varphi\in \core$,
    $$Q_{p,\mathcal{A},V}[\varphi]\geq \int_{\Omega}W^\ast(x)\left(|\nabla\varphi|_{\mathcal{A}}^{p}+|\varphi|^{p}\right)\dx;$$
    \item[$(3)$] for every nonempty open set~$U\Subset\Omega$ there exists~$c_{U}>0$ such that for all~$\varphi\in \core$,
    $$Q_{p,\mathcal{A},V}[\varphi]\geq c_{U}\left(\int_{U}|\varphi|\dx\right)^{p};$$
    \item[$(4)$] the~$(\mathcal{A},V)$-capacity of all closed balls~$B\Subset\Omega$ with positive radii in~$\Omega$ is positive;
    \item[$(4')$] the~$(\mathcal{A},V)$-capacity of some closed ball~$B\Subset\Omega$ with a positive radius in~$\Omega$ is positive.

 \medskip
Furthermore, in the case of~$p<n$,~$(1)$ holds if and only if
\item[$(5)$]  there exists a positive continuous function~$\tilde{W}$ in~$\Omega$ such that the following weighted Hardy–Sobolev–Maz’ya inequality holds true:
$$Q_{p,\mathcal{A},V}[\varphi]\geq \left(\int_{\Omega}\tilde{W}(x)|\varphi|^{p^{\ast}}\dx\right)^{p/p^{\ast}} \qquad \forall~\varphi\in\core,$$ where~$p^{\ast}\triangleq pn/(n-p)$ is the critical Sobolev exponent.
\end{enumerate}
\end{Thm}
\section{Positive solution of minimal growth}\label{minimal}
This section concerns the removability of an isolated singularity, the existence of positive solutions of minimal growth in a neighborhood of infinity in $\Omega$, and their relationships with the criticality or subcriticality. We also study the minimal decay of Hardy-weights.
 \subsection{Removability of isolated singularity}
 In this subsection, we consider the removability of an isolated singularity (see also \cite{Fraas, Serrin1964, Regev25} and references therein).
 \begin{lemma}\label{newev}
    Fix $x_0\in \Omega$. Denote by~$B_r\triangleq B_{r}(x_0)$ the open ball of the radius $r>0$ centered at $x_0$. Suppose that $\mathcal{A}$ satisfies Assumption~\ref{ass8}, and let $V\in M^{q}(p;B_R)$ for some  $R>0$ with~$B_{R}\Subset\Omega$.
Then there exists~$R_1\in(0,R)$ such that $\lambda_{1}(Q_{p,\mathcal{A},V};B_r)>0$ for all $0<r<R_1$.
 \end{lemma}
 \begin{proof}
 By \cite[Theorem 13.19]{Leoni}, for all~$0<r<R$ and~$u\in W^{1,p}_{0}(B_{r})\setminus\{0\}$, we have the lower bound~$\Vert \nabla u\Vert^{p}_{L^{p}(B_{r})}\geq C(n,p)r^{-p}\Vert u\Vert^{p}_{L^{p}(B_{r})}.$ Let $\gd= \alpha_{B_{R}}/2$. Then by the Morrey-Adams theorem, for all~$0\!<\!r\!<\!R$ and $u\!\in\! W^{1,p}_{0}(B_{r})\setminus\{0\}$ with
 $\Vert u\Vert^{p}_{L^{p}(B_{r})}=1$, we get
 \begin{eqnarray*}
  \lambda_{1}(Q_{p,\mathcal{A},V};B_{r})&\geq &  \int_{B_{r}}|\nabla u|_{\mathcal{A}}^{p}\dx+\int_{B_{r}}V|u|^{p}\dx
   \geq \alpha_{B_{R}}\int_{B_{r}}|\nabla u|^{p}\dx+\int_{B_{r}}V|u|^{p}\dx\\
   &\geq& \delta C(n,p)r^{-p}-\frac{C(n,p,q)}{\delta^{n/(pq-n)}}\Vert V\Vert^{n/(pq-n)}_{M^{q}(p;B_{R})} \,.
   \end{eqnarray*}
 Thus, for all sufficiently small radii~$r>0$, the principal eigenvalue $\lambda_{1}(Q_{p,\mathcal{A},V};B_{r})>0$.
 \end{proof}
 The following theorem can be proved by essentially the same arguments as those of \cite[Theorems 5.4]{Pinchover}, and therefore it is omitted.
  \begin{Thm}\label{singularity}
    Assume that~$p\leq n$,~$x_{0}\in\Omega$, $\mathcal{A}$ satisfies Assumption~\ref{ass8}, and $V\in M^{q}_{\loc}(p;\Omega)$. 
    Consider a positive solution~$u$ of~$Q'_{p,\mathcal{A},V}[w]=0$ in a punctured neighborhood $U_{x_0}$ of $x_0$. If $u$ is bounded in some punctured neighborhood of~$x_{0}$, then~$u$ can be extended to a nonnegative solution in $U_{x_0}$. Otherwise,~$\displaystyle{\lim_{x\rightarrow x_{0}}}u(x)=\infty.$
  \end{Thm}
\subsection{Positive solutions of minimal growth}
In this subsection, we study positive solutions of minimal growth at infinity in $\Gw$, a notion that was introduced by Agmon in \cite{Agmon} for second-order linear elliptic operators and was later extended to the quasilinear case \cite{Tintarev} and graphs \cite{Keller}. In particular, we give a further characterization of criticality in terms of global minimal positive solutions.
\subsubsection{Positive solutions of minimal growth}
We call a compact subset~$K$ of~$\Omega$ \emph{admissible} if~$K$ is the closure of a (nonempty) Lipschitz domain with connected boundary.
\begin{lemma}
For every Lipschitz domain~$\omega_{0}\Subset\Omega$,~$\overline{\omega_{0}}$ is admissible if and only if for some (all)  domain(s)~$\omega$ with~$K\subseteq\omega \Subset\Omega$,~$\omega\setminus K$ is a domain. In particular, for every admissible compact subset~$K$ of~$\Omega$ and  every Lipschitz exhaustion~$\{\omega_{i}\}_{i\in\mathbb{N}}$ of~$\Omega$, it holds that  $\omega_{i}\setminus K$ is a domain for all sufficiently large~$i\in\mathbb{N}$.
\end{lemma}
\begin{proof}
Consider an arbitrary Lipschitz domain~$\omega_{0}\Subset\Omega$.

If there exists a domain~$\omega$ with~$\omega_{0}\Subset\omega\Subset\Omega$ such that $\omega\setminus \overline{\omega_{0}}$ is a domain, then it is easy to check that~$\R^{n}\setminus\overline{\omega_{0}}$ is path-connected and hence connected. It follows that~$\R^{n}\setminus\omega_{0}=\overline{\R^{n}\setminus\overline{\omega_{0}}}$ is connected. By the theorem of \cite{connect},~$\partial\omega_{0}$ is connected. Hence, $\overline{\omega_{0}}$ is admissible. 

Now suppose that~$\overline{\omega_{0}}$ is admissible. Then~$\partial\omega_{0}$ is connected. Note that~$\partial\omega_{0}$ is locally path-connected because~$\omega_{0}$ is a Lipschitz domain. Then~$\partial\omega_{0}$ is path-connected. Pick any domain~$\omega$ with~$\omega_{0}\Subset\omega\Subset\Omega$. By means of finite covering, it is a simple matter to check that~$\omega\setminus\overline{\omega_{0}}$ is path-connected and hence a domain.
\end{proof}
The following definition of positive solutions of minimal growth is a modification of \cite[Definition 7.3]{HPR}. 
  \begin{Def}
 \emph{Let $\mathcal{A}$ satisfy Assumption~\ref{ass8} and let $V\in M^{q}_{\loc}(p;\Omega)$. Let $K_{0}$ be a compact subset of~$\Omega$ such that~$\Omega\setminus K_{0}$ is a domain. A positive solution~$u$ of~$Q'_{p,\mathcal{A},V}[w]=0$ in~$\Omega\setminus K_{0}$ is called a \emph{positive solution of minimal growth in a neighborhood of infinity} in $\Omega$ if for all admissible compact subsets~$K$ of~$\Omega$ with~$K_{0}\Subset \mathring{K}$, and all positive solutions~$v\in C\left(\Omega\setminus \mathring{K}\right)$ of~$Q'_{p,\mathcal{A},V}[w]=g\geq 0$  in~$\Omega\setminus K$ for some $g\in M^{q}_{\loc}(p;\Omega\setminus K)$ such that $u\leq v$ on~$\partial K$, it holds that~$ u\leq v$ in~$\Omega\setminus K$. For such a positive solution~$u$, we write~$u\in \mathcal{M}_{\Omega;K_{0}}=\mathcal{M}_{\Omega;K_{0}}^{\mathcal{A},V}$. 
 If~$K_{0}=\emptyset$, then $u$ is called a \emph{global minimal positive solution} of~$Q'_{p,\mathcal{A},V}[w]=0$ in~$\Omega$.}
  \end{Def}
Next we define a slightly stronger local Morrey space~$\widetilde{M}^{q}_{\loc}(p;\Omega)$. For potentials in this space, solutions of~$Q'[u]=0$ with zero boundary value on bounded Lipschitz domains are H\"older continuous up to the boundary (see \cite[Theorem 2.3]{Byun}), which will be used when we apply the weak comparison principle. See also \cite[Definition 2.11]{Hou2}. \begin{definition}\label{strm}
\emph{A function~$V\in M^{q}_{\loc}(p;\Omega)$ is in~$\widetilde{M}^{q}_{\loc}(p;\Omega)$ if~$q>n$ when~$p<n$, and when~$p=n$, for some~$\theta\in (n-1,n)$ and all domains~$\omega\Subset\Omega$,
$$\sup_{\substack{y\in\omega\\0<r<\diam(\omega)}}\frac{1}{r^{\theta}}\int_{B_{r}(y)\cap\omega}|V|\dx<\infty.$$
}
\end{definition}
  \begin{Thm}\label{grcon}
  Let $Q_{p,\mathcal{A},V}\geq 0$ in~$\core$ with~$\mathcal{A}$ satisfying Assumption~\ref{ass8}, and let $V\in \widetilde{M}^{q}_{\loc}(p;\Omega)$. Then for every~$x_{0}\in\Omega$, the equation~$Q'_{p,\mathcal{A},V}[w]=0$ has a solution~$u\in \mathcal{M}_{\Omega;\{x_{0}\}}$.
  \end{Thm}
  \begin{proof}
   Let~$\{\omega_{i}\}_{i\in\mathbb{N}}$ be a Lipschitz exhaustion of~$\Omega$ with~$x_{0}\in\omega_{1}$. We define the inradius of~$\omega_{1}$ as $r_{1}\triangleq\sup_{x\in\omega_{1}}\mathrm{d}(x,\partial\omega_{1})$, and  consider  the open sets~$U_{i}\triangleq\omega_{i}\setminus\overline {B_i}\triangleq\omega_{i}\setminus \overline{B_{r_{1}/(i+1)}(x_{0})},$ for $i\in\mathbb{N}$. Fix a point~$x_{1}\in U_{1}$. Note that~$\{U_{i}\}_{i\in\mathbb{N}}$ is an exhaustion of~$\Omega\setminus\{x_{0}\}.$ Pick a sequence of nonnegative functions $f_{i}\in C^{\infty}_{c}\left(B_{i}(x_{0})\setminus \overline{B_{i+1}(x_{0})}\right)\setminus\{0\}$, for all~$i\in\mathbb{N}$. The principal eigenvalue
   	$$\lambda_{1}\left(Q_{p,\mathcal{A},V+1/i};U_{i}\right)>0,$$
   	because~$Q_{p,\mathcal{A},V}$ is nonnegative in~$\Omega$. Then, by virtue of Theorem \ref{maximum}, for every~$i\in\mathbb{N}$, there exists a positive solution~$v_{i}\in W^{1,p}_{0}(U_{i})$ of~$Q'_{p,\mathcal{A},V+1/i}[u]=f_{i}$ in~$U_{i}$. The Harnack convergence principle yields a subsequence of
  ~$\big\{u_{i}\triangleq v_{i}(x)/v_{i}(x_{1})\big\}_{i\in\mathbb{N}}$ converging locally uniformly in $\Omega\setminus\{x_{0}\}$ to a positive solution~$u$ of $Q'_{p,\mathcal{A},V}[u]=0$ in $\Omega\setminus\{x_{0}\}$.

  We claim that~$u\in\mathcal{M}_{\Omega;\{x_{0}\}}$. Consider any admissible compact subset~$K$ of $\Omega$ with~$x_{0}\in\mathring{K}$, and any positive solution~$v\in C\left(\Omega\setminus \mathring{K}\right)$ of~$Q'_{p,\mathcal{A},V}[w]=g$ in~$\Omega\setminus K$ satisfying~$u\leq v$ on~$\partial K$, where $g\in M^{q}_{\loc}(p;\Omega\setminus K)$ is nonnegative. 
  For an arbitrary~$\delta>0$, there exists~$i_{K}\in\mathbb{N}$ such that~for all~$i\geq i_{K}$, $\supp{f_{i}}\Subset K$,~$\omega_{i}\setminus K$ is a domain, and~$u_{i}\leq (1+\delta)v$ on~$\partial\left(\omega_{i}\setminus K\right)$. The weak comparison principle (Theorem~\ref{thm_wcp}) gives~$u_{i}\leq (1+\delta)v$ in~$\omega_{i}\setminus K$. Then by letting~$i\rightarrow\infty$ and then $\delta\rightarrow 0$, we obtain~$u\leq v$ in~$\Omega\setminus K$.
  \end{proof}
  \begin{Def}
   \emph{A function~$u\in \mathcal{M}_{\Omega;\{x_{0}\}}$ is called a \emph{minimal positive Green function of~$Q'_{p,\mathcal{A},V}$ in~$\Omega$ with singularity} at~$x_{0}$, if~$u$ admits a nonremovable singularity at~$x_{0}$.  We denote such a Green function by~$G^{\Omega}_{\mathcal{A},V}(x,x_{0})$.}
  \end{Def}
  \begin{Rem}
  \emph{See \cite{PinchoverGreen, PinchoverGreen2,Pinchoverlinear} for more on minimal positive Green functions of linear elliptic operators of the second order.}
  \end{Rem}
  \subsubsection{Further characterization of criticality}\label{tacit}
 We characterize the criticality and subcriticality of $Q_{p,\mathcal{A},V}$ in terms of the existence of a global minimal  positive solution and give sufficient conditions for the subcriticality by virtue of the existence of a Green function.
  \begin{Thm}
Let~$\mathcal{A}$ satisfy assumptions~\ref{ass8} and \ref{ass2}, and let~$V\in \widetilde{M}^{q}_{\loc}(p;\Omega)$. {In addition, for $1<p<2$ suppose that Assumption~\ref{ngradb} holds.}
 Consider the nonnegative functional~$Q_{p,\mathcal{A},V}$. Then~$Q_{p,\mathcal{A},V}$ is subcritical in $\Gw$ if and only if the equation $Q'_{p,\mathcal{A},V}[u]=0$ does not admit a global minimal positive solution in~$\Omega$. Moreover, a ground state of~$Q'_{p,\mathcal{A},V}[u]=0$ in~$\Omega$ is a global minimal positive solution of~$Q'_{p,\mathcal{A},V}[u]=0$ in~$\Omega$.
\end{Thm}
\begin{proof}
 The proof is similar to that of \cite[Theorem 5.9]{Pinchover} and hence omitted.
\end{proof}
    \begin{Thm}
   Let~$\mathcal{A}$ satisfy assumptions~\ref{ass8} and \ref{ass2}, and let $V\in \widetilde{M}^{q}_{\loc}(p;\Omega)$. {In addition, for $1<p<2$ suppose that Assumption~\ref{ngradb} holds.}
  Assume that the  functional~$Q_{p,\mathcal{A},V}$ is nonnegative in $\Gw$, and fix $u\in \mathcal{M}_{\Omega;\{x_{0}\}}$ for some $x_{0}\in\Omega$.
  \begin{enumerate}
    \item[$(1)$] If~$p\leq n$ and~$u$ has a nonremovable singularity at~$x_{0}$, then~$Q_{p,\mathcal{A},V}$ is subcritical in~$\Omega$.
    \item[$(2)$] If~$p> n$,~$u$ has a nonremovable singularity at~$x_{0},$ and~$\lim_{x\rightarrow x_{0}}u(x)=c$ for some positive constant~$c,$ then~$Q_{p,\mathcal{A},V}$ is subcritical in~$\Omega$.
  \end{enumerate}
\end{Thm}
\begin{proof}
  The proof is similar to that of \cite[Theorem 5.10]{Pinchover} and hence omitted.
\end{proof}
\subsection{How large can Hardy-weights be?}\label{tacit2}
 The following theorem is a generalization of \cite[theorems~3.1 and 3.2]{Kovarik}.
\begin{Thm}
	Let~$\mathcal{A}$ satisfy assumptions~\ref{ass8} and \ref{ass2} and let~$V\in \widetilde{M}^{q}_{\loc}(p;\Omega)$. In addition, for all~$1<p<\infty$ suppose that Assumption~\ref{ngradb} holds.
  Assume that~$Q_{p,\mathcal{A},V}$ is nonnegative in~$\Omega$. For a compact $K\subseteq\Omega$ such that~$\Omega\setminus K$ is a domain, let $\phi\in W^{1,p}_{\loc}(\Gw\setminus K)$ be a positive solution of the equation $Q'_{p,\mathcal{A},V}[u]=0$ in  $\Gw\setminus K$ of minimal growth\vspace{0.5mm} in a neighborhood of infinity in $\Gw$. Let~$\mathcal{K}$ be an admissible compact subset of~$\Omega$ such that~$K\Subset \mathring{\mathcal{K}}$.
	Then for every Hardy-weight $W$ of $Q_{p,\mathcal{A},V}$ in $\Gw\setminus K$, we have $$\int_{\mathcal{K}^{c}}W|\phi|^{p}\dx<\infty.$$
\end{Thm}
\begin{proof}
	Let $K\Subset\mathring{\mathcal{K}}\Subset\Omega$, and let $\tilde V\in C_c^\infty(\mathring{\mathcal{K}})$ be a nonnegative function such that $Q_{p,\mathcal{A},V-\tilde V}$ is critical in $\Gw$.  There exists a null sequence $\{\varphi_{k}\}_{k\in\mathbb{N}}\subseteq \core$ for $Q_{p,\mathcal{A},V-\tilde V}$ in $\Gw$ converging locally uniformly to its ground state $\vgf$. Then $\vgf_{k}\geq 0$,~$\Vert\vgf_{k}\Vert_{L^{p}(K)}=1$, and~$\lim_{k\rightarrow\infty}Q_{p,\mathcal{A},V-\tilde V}[\vgf_{k}]\!=\!0$.
		 Let $f\in C^{1}(\Omega)$  satisfy $0\leq f\leq 1$,~$f|_{K} =  0$, $f|_{\mathcal{K}^{c}}=1$, and~$|\nabla f(x)|_{\mathcal{A}}\leq C_{0}$ for some constant~$C_{0}$ and all~$x\in\Omega$.  Then~$Q_{p,\mathcal{A},V}[f\vgf_{k}]\geq \int_{K^{c}}W|f\vgf_{k}|^{p}\dx\geq \int_{\mathcal{K}^{c}}W|\vgf_{k}|^{p}\dx$. Moreover,
	\begin{eqnarray*}
		\int_{\mathcal{K}^{c}}W|\vgf_{k}|^{p}\dx \!&\leq&\! Q_{p,\mathcal{A},V}[f\vgf_{k}]\!=\!\int_{\mathcal{K}^{c}} \!\!\! (|\nabla\vgf_{k}|_{\mathcal{A}}^{p} \!+\! V|\vgf_{k}|^{p})\dx+\int_{\mathcal{K}\setminus K} \!\!\!( |\nabla(f\vgf_{k})|_{\mathcal{A}}^{p}+V|f\vgf_{k}|^{p})\!\dx\\
		\!&\leq&\!\!Q_{p,\mathcal{A},V-\tilde V}[\vgf_{k}]+2\!\int_{\mathcal{K}}\!(|V|+\tilde V) |\vgf_{k}|^{p}\!\dx +C\|\vgf_k\|^p_{W^{1,p}(\mathcal{K})}.
	\end{eqnarray*}
Since the null-sequence $\{\varphi_{k}\}$ is locally bounded in $L^\infty(\Omega)$ and~$W^{1,p}(\Gw)$, it follows that$$\int_{\mathcal{K}^{c}}W|\vgf_{k}|^{p}\dx < C_1.$$ Consequently, the Fatou Lemma implies that $\int_{\mathcal{K}^{c}}W|\vgf|^{p}\dx\leq C_1$. Note that the ground state $\vgf$ is a positive solution of $Q'_{p,\mathcal{A},V}[u]=0$ of minimal growth at infinity of~$\Omega$. Hence, $\gf\asymp\vgf$ in $\mathcal{K}^c$. Thus, $\int_{\mathcal{K}^{c}}W|\gf|^{p}\dx< \infty$.
\end{proof}

\subsection*{Data Availability Statement}
Data sharing is not applicable to this article as no datasets were generated or analyzed during the current study.

\subsection*{Declarations}
\begin{itemize}
\item
This paper is based on the
{thesis of the first author for the M.Sc. degree in Mathematics conferred by the  Technion-Israel Institute of Technology under the supervision of Professors Yehuda Pinchover and Antti Rasila and  on his subsequent Ph.D. research at the Technion.}
\item Y.H. and A.R. gratefully acknowledge the generous financial help of NNSF of China (No. 11971124) and NSF of Guangdong Province (No. 2021A1515010326). {Y.H. and Y.P.  acknowledge the support of the Israel Science Foundation (grant 637/19) founded by the Israel Academy of Sciences and Humanities.} {Y.H. is grateful to the Technion for supporting his study.}
\item
The authors have no competing interests to declare that are relevant to the content of this article.
\end{itemize}
 {\small
{}
}

\end{document}